\documentclass[10pt,letterpaper]{article}
\usepackage[margin=1in]{geometry}
\usepackage[utf8]{inputenc}
\usepackage[T1]{fontenc}
\usepackage{lmodern}
\usepackage{amsmath,amssymb,amsfonts,amsthm}
\usepackage{bm}
\usepackage{graphicx}
\usepackage{float}
\usepackage{booktabs,multirow,tabularx}
\usepackage{subfigure}
\usepackage[justification=centering]{caption}
\usepackage{xcolor}
\usepackage[hyphens]{url}
\usepackage[hidelinks,hypertexnames=false]{hyperref}
\numberwithin{equation}{section}
\newtheorem{theorem}{Theorem}[section]
\newtheorem{lemma}[theorem]{Lemma}

\theoremstyle{definition}

\newtheorem{assumption}[theorem]{Assumption}
\theoremstyle{remark}
\newtheorem{remark}[theorem]{Remark}
\newenvironment{keywords}{\par\smallskip\noindent\textbf{Keywords.}\ }{\par\smallskip}
\newenvironment{MSCcodes}{\par\smallskip\noindent\textbf{MSC codes.}\ }{\par\smallskip}

\title{A second-order energy-stable sequential regularization scheme for the Navier-Stokes equations}
\author{Zhaoyang Wang\thanks{Department of Mathematics, Beijing Normal University, Zhuhai, Guangdong 519087, China (\href{mailto:zhaoyang584520@163.com}{\nolinkurl{zhaoyang584520@163.com}}).}
\and Ping Lin\thanks{Corresponding author. Division of Mathematics, University of Dundee, Dundee DD1 4HN, United Kingdom (\href{mailto:p.lin@dundee.ac.uk}{\nolinkurl{p.lin@dundee.ac.uk}}).}}
\date{}
\hypersetup{
 pdftitle={A second-order energy-stable sequential regularization scheme for the Navier-Stokes equations},
 pdfauthor={Zhaoyang Wang and Ping Lin}
}
\begin{document}

\maketitle

\begin{abstract}
In this paper, we construct a second-order decoupled scheme for the Navier-Stokes equations based on the idea of the auxiliary variable and the sequential regularization method (SRM). The proposed scheme enjoys unconditional energy stability, which is generally unavailable for the existing SRM, while preserving its key advantage of avoiding artificial pressure boundary conditions and maintaining the original velocity boundary conditions. We establish a viscosity-independent low-order stability estimate and a higher-order stability estimate for the proposed scheme. Furthermore, we derive global-in-time second-order error estimates in two dimensions and corresponding local-in-time estimates in three dimensions, without imposing any time-step restrictions. A series of numerical experiments confirm the convergence and stability, and demonstrate improved accuracy and incompressibility control compared with the projection method and the conventional SRM. In particular, both high-Reynolds-number and three-dimensional numerical simulations further demonstrate the robustness and reliability of the proposed method for complex incompressible flows.
\end{abstract}



\section{Introduction}
\label{section1}
In this paper, we consider numerical approximation of the incompressible Navier-Stokes equations:
\begin{subequations}\label{the1-1}
	\begin{align}
&\frac{\partial \bm{u}}{\partial t}+(\bm{u} \cdot \nabla) \bm{u}-\nu \Delta \bm{u}+\nabla p=\bm{f} \quad \text{in} \ \Omega\times(0,T), \label{the1-1a}\\
&\nabla\cdot \bm{u}=0, \ \bm{u}|_{t=0}=\bm{u}_0, \quad \text{in} \ \Omega\times(0,T), \label{the1-1b}
\end{align}
\end{subequations}
where $\Omega$ is a bounded domain in $\mathbb{R}^d \ (d=2, 3)$ with no-slip boundary condition for simplicity in the analysis. $\bm{u}$ and $p$ are the fluid velocity and pressure, $\bm{f}$ is an external force, and $\nu>0$ is the viscosity.

It is well known that the Navier–Stokes equations are the most fundamental equations for describing fluid motion. Over the past half-century, the numerical approximation of the Navier-Stokes equations has been extensively studied and remains an area of active research. The mainstream numerical methods include coupled pressure-velocity solvers (for example, linearization schemes \cite{brezzi2012mixed, john2006comparison}) and decoupled methods. The computational cost of directly solving the coupled problem is generally expensive, since the pressure acts as a Lagrange multiplier of the incompressibility constraint, which leads to a saddle point problem required to be solved at each time step. Moreover, the treatment of the nonlinear term (convective term), with consideration of numerical stability, constitutes an additional computational difficulty. The decoupled method originated from the projection method proposed by Chorin \cite{chorin1968numerical} and Temam \cite{temam1969approximation}, a method that splits the saddle point system into two small linear problems and is therefore computationally very efficient. However, solving the problem using the projection method requires imposing the artificial boundary condition on the pressure equation, which prevents the pressure from achieving full-order accuracy and results in the final velocity not exactly satisfying the boundary conditions. Subsequently, extensive work has been devoted to improving the projection method or developing alternative decoupled approaches, including correction schemes \cite{goda1979multistep, van1986second, guermond2003velocity, yang2026energy}, rotational incremental schemes \cite{timmermans1996approximate, guermond2001quelques} and consistent splitting schemes \cite{weinan2003gauge}. Although some of these methods overcome certain shortcomings of the projection scheme, they still exhibit deficiencies. For example, they either rely on the artificial boundary condition or fail to strictly satisfy the incompressibility constraint. For an overview of these decoupled methods, see \cite{guermond2006overview}. 

The penalty method \cite{temam1968methode, shen1995error} achieves velocity–pressure decoupling by relaxing the incompressibility constraint. It allows the pressure to be eliminated while preserving the original velocity boundary conditions. However, improving the accuracy of the approximation generally requires a smaller penalty parameter, which may lead to a stiff numerical system. As an improvement of the penalty method, the sequential regularization method \cite{lin1997sequential, lu2008analysis} can effectively alleviate the stiffness and more accurately approximate the divergence-free condition by using a sequence of regularization steps. In particular, instead of introducing a single small penalty parameter, the SRM gradually enforces the incompressibility constraint through a series of regularized subproblems, which improves the accuracy of the numerical approximation.

It is important to ensure stability in decoupled schemes for the Navier–Stokes equations. At the continuous level, the incompressible Navier–Stokes equations possess a fundamental energy stability, which yields a uniform bound for the velocity in $L^\infty(0,T;L^2)$. However, the existing SRM generally lacks an energy-stable structure, and the available stability analyses yield estimates whose constants depend explicitly on $\nu$ and may deteriorate as $\nu$ becomes small \cite{lin1997sequential, lu2008analysis}. Therefore, the resulting stability bounds are not uniform with respect to the viscosity, which seems to limit the robustness of schemes constructed based on the SRM for high Reynolds number flows. This difficulty is closely related to the sequential regularization mechanism of the SRM. In order to enforce the incompressibility constraint, the pressure approximation $p_{s-1}$ obtained from the preceding regularization stage is incorporated into the momentum equation at the current stage. Since the intermediate velocity fields are not exactly divergence-free, the corresponding pressure contribution does not vanish in the energy balance as it does for the Navier-Stokes equations. Consequently, the stability analysis has to control not only the velocity but also the pressure and the associated divergence defect. In the existing analysis, such control is typically coupled with the viscous term, which leads to stability constants involving $\nu$ and therefore deteriorating in the small-viscosity regime.

The above observation motivates us to develop a modified SRM formulation that retains the main advantages of the sequential regularization strategy while recovering an energy-stable structure. In recent years, Shen and his collaborators \cite{shen2018scalar, shen2019new, huang2020highly, huang2022new} proposed the scalar auxiliary variable (SAV) approach for dissipative systems, which provides an effective framework for constructing unconditionally energy-stable numerical schemes. Li et al. \cite{li2022new} applied the SAV approach to construct a pressure-correction scheme for the Navier-Stokes equations and derived rigorous error estimates in two dimensions. Wang et al. \cite{wang2026efficient} proposed a second-order artificial compressibility-SAV for the Navier-Stokes equations and provided convergence analysis in two dimensions. The stability bounds and error estimates for these schemes are limited to low-order norms ($\ell^\infty(L^2)\cap\ell^2(H^1)$ for the velocity). Huang and Shen \cite{huang2023stability} proposed a new second-order consistent splitting scheme for the Navier-Stokes equations based on the generalized SAV approach, for which they established higher-order stability and rigorous error estimates under a time-step restriction.

In this work, we develop a robust second-order sequential regularization scheme with an auxiliary variable (SR-AV) for the incompressible Navier-Stokes equations. By introducing a time-dependent auxiliary variable to compensate for the coupling between the pressure from the preceding regularization level and the divergence defect of the current velocity, the new formulation recovers an energy dissipation structure while preserving the key computational advantages of the conventional SRM, namely velocity-pressure decoupling and the avoidance of artificial pressure boundary conditions. At the continuous level, we establish stability estimates for the SR-AV formulation and prove that it preserves the successive convergence property of the conventional SRM, with the velocity and pressure errors decreasing as $O(\epsilon^s)$ at the $s$-th regularization level. At the discrete level, we construct a second-order time-stepping scheme and prove its unconditional energy stability. In particular, we establish a viscosity-independent low-order stability estimate for the velocity in $\ell^\infty(L^2)$, together with viscosity-dependent higher-order stability estimates in $\ell^\infty(H^1)\cap\ell^2(H^2)$. We further derive a global-in-time optimal error estimate in two dimensions and a local-in-time error estimate in three dimensions without any time-step constraints. It should be noted that the proposed SR-AV scheme captures the evolution of three-dimensional vortical wake structures in high-Reynolds-number simulations of flow past an aircraft geometry based on the NASA Common Research Model.

The rest of this paper is organized as follows. In Section \ref{section2}, we introduce the SR-AV system and establish its stability and convergence at the continuous level. In Section \ref{section3}, we propose a second-order SR-AV scheme and prove its stability. In Section \ref{section4}, we carry out a rigorous error analysis for the SR-AV scheme. Numerical experiments are presented in Section 5, and Section 6 concludes the paper.

\subsection*{Notation}
For a domain $\Omega$ in $\mathbb{R}^{d} \ (d=2,3)$ and $1\leq p\leq\infty$, we use the standard notation for the Banach space $L^{p}(\Omega)$ and the Sobolev space $W^{k,p}(\Omega)$, $H^p(\Omega)$, and $H^p_0(\Omega)$. The symbol $(\cdot,\cdot)$ denotes the standard inner product in $L^2$. Throughout this paper, the letter $C$ represents a generic positive constant, with or without subscripts, whose value may vary from one estimate to another. The dependence of $C$ on relevant parameters will be explicitly indicated whenever necessary.

\section{The sequential regularization auxiliary variable (SR-AV) system and its convergence}
\label{section2}
In this section, we formulate a new system to approximate the Navier–Stokes equations (1.1) based on the idea of the auxiliary variable and the SRM. We then analyze the error between the SR-AV system and the original Navier–Stokes system.

\subsection{System reconstruction} Let $q(t)$ be a time-dependent auxiliary variable with $q|_{t=0}=1$. We introduce the SRM: given an initial guess $p_0$ and a positive constant $\alpha$, for $s=1, 2, 3, ...$, find the solution $(\bm{u}_s, p_s, q_s)$ of the following system:
\begin{subequations}\label{the2-1}
	\begin{align}
&\frac{\partial \bm{u}_s}{\partial t}+(\bm{u}_s \cdot \nabla) \bm{u}_s+\frac{1}{2}(\nabla\cdot \bm{u}_s)\bm{u}_s-\nu \Delta \bm{u}_{s}+\nabla p_s=\bm{f}, \label{the2-1a}\\
&\frac{d}{dt}q_s(t)=-\int_\Omega p_{s-1}(\nabla\cdot \bm{u}_s) \ d\bm{x}, \label{the2-1b} \\
&\frac{\partial (\nabla\cdot \bm{u}_{s})}{\partial t}+\alpha \nabla\cdot \bm{u}_{s}=-\epsilon (p_s-q_sp_{s-1}), \quad \bm{u}_s|_{t=0}=\bm{u}_0, \label{the2-1c}
\end{align}
\end{subequations}
where $\epsilon$ is a small parameter. We consider the homogeneous Dirichlet boundary condition. We remark that (2.1a) and (2.1c) with q=1 are the conventional SRM ([16]) and the left hand form of (2.1c) is called the socalled Baumgarte formulation which ensures the error in divergence-free condition does not accumulate with the time. It can be seen that system (\ref{the2-1}) is not equivalent to the original system since 
\begin{equation}\nonumber
\begin{split}
\int_{\Omega} p_{s-1}(\nabla\cdot \bm{u}_s) \ d\bm{x}\not\equiv 0
\end{split}
\end{equation}
implies $q(t)\not\equiv1$. The auxiliary variable $q_s$ is introduced to compensate for the coupling term between the pressure from the previous regularization level and the divergence of the current velocity. More precisely, although the introduction of $q_s$ slightly modifies the conventional SRM, the resulting formulation possesses an energy stability property that is generally unavailable for the conventional SRM.

Since the formulated auxiliary problem (\ref{the2-1}) is not equivalent to the Navier–Stokes system (\ref{the1-1}), we will provide its error estimate below. As preliminaries, we begin by recalling the following inequalities to be used in the proof.

The Poincar\'e inequality,
\begin{equation}\nonumber
\begin{split}
\Vert \bm{u}\Vert_{L^2}\leq C_{\Omega} \Vert \nabla\bm{u}\Vert_{L^2}, \quad \forall \bm{u}\in H^{1}_0(\Omega),
\end{split}
\end{equation}
where $C_{\Omega}$ is a constant depending only on $\Omega$ and $H^{1}_0(\Omega)=\left\{\bm{u}\in H^{1}(\Omega): \bm{u}|_{\partial\Omega}=0\right\}$.

We define the trilinear form $B(\cdot, \cdot, \cdot)$ by 
\begin{equation}\nonumber
\begin{split}
B(\bm{u}_1, \bm{u}_2, \bm{u}_3)=\int_\Omega \left(\left(\bm{u}_1\cdot \nabla\right)\bm{u}_2\cdot\bm{u}_3+\frac{1}{2}(\nabla \cdot\bm{u}_1)\bm{u}_2\cdot \bm{u}_3\right) \ d\bm{x}.
\end{split}
\end{equation}
By using H\"{o}lder's, Agmon and Sobolev inequalities \cite{lu2008analysis}, we have that for $d\leq 3$,
\begin{equation}\label{the2-2}
\begin{split}
\left|B(\bm{u}_1, \bm{u}_2, \bm{u}_3)\right|\leq 
\left\{
\begin{aligned}
&C_\Omega\Vert \bm{u}_1\Vert_{H^1}\Vert \bm{u}_2\Vert_{H^1}\Vert \bm{u}_3\Vert_{H^1},\\
&C_\Omega\Vert \bm{u}_1\Vert_{H^1}\Vert \bm{u}_2\Vert_{H^2}\Vert \bm{u}_3\Vert_{L^2},\\
&C_\Omega\Vert \bm{u}_1\Vert_{H^2}\Vert \bm{u}_2\Vert_{H^1}\Vert \bm{u}_3\Vert_{L^2}.\\
\end{aligned}
\right.
\end{split}
\end{equation}
In addition,
\begin{equation}\label{the2-3}
\begin{split}
&\left|B(\bm{u}_1, \bm{u}_2, \bm{u}_3)\right|\leq C_\Omega\left(\Vert \bm{u}_1\Vert_{L^2}^{1/2}\Vert \bm{u}_1\Vert_{H^1}^{1/2}\Vert \bm{u}_2\Vert_{H^1}^{1/2}\Vert \bm{u}_2\Vert_{H^2}^{1/2}\Vert \bm{u}_3\Vert_{L^2}+\Vert\bm{u}_1\Vert_{H^1}\Vert \bm{u}_2\Vert_{L^2}^{1/2}\Vert \bm{u}_2\Vert_{H^2}^{1/2}\Vert \bm{u}_3\Vert_{L^2}\right), \ d=2,\\
&\left|B(\bm{u}_1, \bm{u}_2, \bm{u}_3)\right|\leq C_\Omega\Vert \bm{u}_1\Vert_{H^1}\Vert \bm{u}_2\Vert_{H^1}^{1/2}\Vert \bm{u}_2\Vert_{H^2}^{1/2}\Vert \bm{u}_3\Vert_{L^2}, \ d=3.
\end{split}
\end{equation}

We define the time-dependent operator $\mathcal{A}\bm{u}=-\nu\Delta \bm{u}-\frac{1}{\epsilon}\nabla(\nabla\cdot \bm{u})_t-\frac{\alpha}{\epsilon}\nabla\nabla\cdot \bm{u}$. Under the compatibility condition $\nabla\cdot\bm{u}_0=0$, the following estimate holds (see Lemma 2.1 in \cite{lin1997sequential}):
\begin{equation}\label{the2-4}
\begin{split}
&\int_0^t \Vert \Delta \bm{u}(s)\Vert_{L^2}^2 \ ds\leq C_{\nu}\int_0^t \Vert \mathcal{A}\bm{u}(s)\Vert_{L^2}^2 \ ds.
\end{split}
\end{equation}

We recall the following lemma\cite{huang2023stability, liu2010stable}, which will be used to derive local error estimates for the three-dimensional case.

\begin{lemma}
\label{lemma2-1}
Let $G:(0,\infty)\rightarrow(0,\infty)$ be continuous and increasing, and let $T^*$ satisfy that $0<T^*<\int_{M}^{\infty} dz/G(z)$ for some $M>0$. Suppose that quantities $z_n, w_n\geq 0$ satisfy
\begin{equation}
\begin{split}
z_n+\sum\limits_{i=0}^{n-1}\Delta t w_{i}\leq M+\sum\limits_{i=0}^{n-1}\Delta t G(z_i), \ \forall n\leq n^{*},
\end{split}
\end{equation}
with $n^{*}\Delta t\leq T^*$. Then we have $M+\sum\limits_{i=0}^{n-1}\Delta t G(z_i)\leq C_{*}$, where $C_{*}$ is independent of $\Delta t$.
\end{lemma}

We then present the essential regularity assumptions.
\begin{assumption}
\label{assumption2-1}
Assume that $\bm{u}_0\in H_0^1$ and that it satisfies the compatibility condition $\nabla\cdot \bm{u}_0=0$. In addition, we assume $\bm{f}\in L^\infty(0,T; L^2)$. For the solution $(\bm{u}, p)$ of the Navier-Stokes equations (\ref{the1-1}), it holds that \cite{temam2024navier}
\begin{equation}\label{the2-5}
\begin{split}
\bm{u}\in L^{\infty} (0,T; H^1)\cap L^{2} (0,T; H^2), \ p\in L^2(0,T; H^1).
\end{split}
\end{equation}
\end{assumption}

\subsection{Stability and error estimate of the SR-AV formulation}
We provide an error estimate between the problem (\ref{the2-1}) and the original problem (\ref{the1-1}). As a preliminary step, we first establish the stability of the SR-AV system.

\begin{lemma}
\label{lemma2-2}
Let $(\bm{u}_s, p_s, q_s)$ be the solution of (\ref{the2-1}). Assume that $\bm{u}_0\in H_0^1(\Omega)$ with
$\nabla\cdot\bm{u}_0=0$, $p_0\in L^2(0,T; H^1)$, and $\bm{f}\in L^\infty(0,T;L^2(\Omega))$. Then, for $d=2, 3$ and any $t\in[0,T]$, there exist positive constants $C_{L23a}$ and $C_{L23b}$ such that
\begin{equation}\label{2-7}
\begin{split}
\Vert \bm{u}_s(t)\Vert_{L^2}^2 + |q_s(t)|^2+\frac{1}{\epsilon} \Vert \nabla\cdot\bm{u}_s(t)\Vert_{L^2}^2+\int_0^t \left(\frac{\alpha}{\epsilon}\Vert\nabla\cdot\bm{u}_s(r)\Vert_{L^2}^2+\nu\Vert \bm{u}_s(r)\Vert_{L^2}^2\right)dr\leq C_{L23a},
\end{split}
\end{equation}
where $C_{L23a}$ is independent of the viscosity $\nu$.

Moreover, for $d=2$, the following global high-order stability estimate holds:
\begin{equation}\label{2-8}
\begin{split}
\Vert \nabla\bm{u}_s(t)\Vert_{L^2}^2
+\int_0^{t}
\left(\Vert \Delta\bm{u}_s(r)\Vert_{L^2}^2+\Vert \nabla p_s(r)\Vert_{L^2}^2\right) \ dr
\leq C_{L23b}, \qquad 0\leq t\leq T.
\end{split}
\end{equation}
For $d=3$, there exists a $T^*>0$ such that (\ref{2-8}) holds for $0\leq t\leq T^*$. Here, $C_{L23b}$ may depend on $\nu$ and the regularization level $s$.
\end{lemma}

\begin{proof}
By eliminating $p_s$ from (\ref{the2-1a}) using (\ref{the2-1c}), taking the $L^2$ inner product of the resulting equation with $\bm{u}_s$, and using $B(\bm{u}_s, \bm{u}_s, \bm{u}_s)=0$, we obtain
\begin{equation}\label{2-9}
\begin{split}
\frac{1}{2}\frac{d}{dt}
\left(
\|\bm{u}_s\|_{L^2}^2
+\frac{1}{\epsilon}
\|\nabla\cdot\bm{u}_s\|_{L^2}^2
\right)
+\nu\|\nabla\bm{u}_s\|_{L^2}^2
+\frac{\alpha}{\epsilon}
\|\nabla\cdot\bm{u}_s\|_{L^2}^2+(q_s\nabla p_{s-1}, \bm{u}_s)
=(\bm{f},\bm{u}_s).
\end{split}
\end{equation}

Multiplying (\ref{the2-1b}) by $q_s$, using integration by parts together with the homogeneous Dirichlet boundary condition, and then combining the resulting identity with (\ref{2-9}), we obtain
\begin{equation}\label{2-10}
\begin{split}
\frac{1}{2}\frac{d}{dt}
\left(
\|\bm{u}_s\|_{L^2}^2+ |q_s|^2
+\frac{1}{\epsilon}
\|\nabla\cdot\bm{u}_s\|_{L^2}^2
\right)
+\nu\|\nabla\bm{u}_s\|_{L^2}^2
+\frac{\alpha}{\epsilon}
\|\nabla\cdot\bm{u}_s\|_{L^2}^2
=(\bm{f},\bm{u}_s).
\end{split}
\end{equation}
In particular, in the absence of external forcing, the reformulated system is energy dissipative. 

Let $E_s(t):=\|\bm{u}_s\|_{L^2}^2+ |q_s|^2+\frac{1}{\epsilon}\|\nabla\cdot\bm{u}_s\|_{L^2}^2$. It follows from (\ref{2-10}) that
\begin{equation}\label{2-11}
\begin{split}
\frac{d}{dt}E_s(t)+\frac{2\alpha}{\epsilon}\Vert \nabla\cdot \bm{u}_s\Vert_{L^2}^2 \leq 2\Vert\bm{f}\Vert_{L^2}\Vert\bm{u}_s\Vert_{L^2}\leq 2\Vert\bm{f}\Vert_{L^2} E_s^{1/2}(t).
\end{split}
\end{equation}
Therefore,
\begin{equation}\label{2-12}
\begin{split}
\frac{d}{dt}E_s^{1/2}(t)\leq \Vert\bm{f}\Vert_{L^2}.
\end{split}
\end{equation}
Integrating (\ref{2-12}) over $(0,t)$ and using $q_s(0)=1$, we obtain
\begin{equation}\label{2-13}
\begin{split}
E_s(t)\leq 2(\Vert\bm{u}_0\Vert_{L^2}^2+1+C_{f,T}):=C_{L23a},
\end{split}
\end{equation}
and $\int_0^t\left(\frac{\alpha}{\epsilon} \Vert \nabla\cdot \bm{u}_s(r)\Vert_{L^2}^2+\nu\Vert \nabla\bm{u}_s(r)\Vert_{L^2}^2\right) dr\leq C_{L23a}$.

We next derive the high-order stability estimate. We first consider the two-dimensional case ($d=2$). By taking the $L^2$ inner product of (\ref{the2-1}) with $\mathcal{A}\bm{u}_s$, we can obtain
\begin{equation}\label{2-14}
\begin{split}
&\frac{1}{2}\frac{d}{dt}
\left(
\nu\|\nabla\bm{u}_s\|_{L^2}^2
+\frac{\alpha}{\epsilon}
\|\nabla\cdot\bm{u}_s\|_{L^2}^2
\right)
+\frac{1}{\epsilon}
\|(\nabla\cdot\bm{u}_s)_t\|_{L^2}^2
+\|\mathcal{A}\bm{u}_s\|_{L^2}^2
\\
&=-B(\bm{u}_s,\bm{u}_s,\mathcal{A}\bm{u}_s)-q_s(\nabla p_{s-1},\mathcal{A}\bm{u}_s)
+(\bm{f},\mathcal{A}\bm{u}_s).
\end{split}
\end{equation}

For any $\delta>0$,
\begin{equation}\label{2-15}
\begin{split}
\left|
q_s(\nabla p_{s-1},\mathcal{A}\bm{u}_s)
\right|+|(\bm{f},\mathcal{A}\bm{u}_s)|
&\leq
|q_s|\|\nabla p_{s-1}\|_{L^2}
\|\mathcal{A}\bm{u}_s\|_{L^2}+\left|(\bm{f},\mathcal{A}\bm{u}_s)\right|\\
&\leq
\delta\|\mathcal{A}\bm{u}_s\|_{L^2}^2
+C\left(\|\nabla p_{s-1}\|_{L^2}^2+\|\bm{f}\|_{L^2}^2\right).
\end{split}
\end{equation}

For the nonlinear term, we have
\begin{equation}\label{2-16}
\begin{split}
|B(\bm{u}_s,\bm{u}_s,\mathcal{A}\bm{u}_s)|
&\leq
C\left(\|\bm{u}_s\|_{L^2}^{1/2}\|\bm{u}_s\|_{H^1}^{1/2}
\|\nabla\bm{u}_s\|_{L^2}^{1/2}
\|\bm{u}_s\|_{H^2}^{1/2}
\|\mathcal{A}\bm{u}_s\|_{L^2}+\|\bm{u}_s\|_{H^1}
\|\bm{u}_s\|_{L^2}^{1/2}
\|\bm{u}_s\|_{H^2}^{1/2}
\|\mathcal{A}\bm{u}_s\|_{L^2}\right)\\
&\leq \delta\|\mathcal{A}\bm{u}_s\|_{L^2}^2+\delta\|\Delta \bm{u}_s\|_{L^2}^2+C\Vert \nabla\bm{u}_s\Vert_{L^2}^4.
\end{split}
\end{equation}

Integrating (\ref{2-14}) over $(0,t)$, using (\ref{the2-4}), (\ref{2-15}), (\ref{2-16}), and choosing $\delta$ sufficiently small, we obtain
\begin{equation}\label{2-17}
\begin{split}
\nu\Vert \nabla \bm{u}_s \Vert_{L^2}^2
+C_\delta \int_0^t
\|\mathcal{A}\bm{u}_s(r)\|_{L^2}^2 \ dr
&\leq
\nu\Vert \nabla \bm{u}_s(0)\Vert_{L^2}^2+C_f+C\int_0^t \|\nabla p_{s-1}(r)\|_{L^2}^2 \ dr\\
&+C\int_0^t
\|\nabla\bm{u}_s(r)\|_{L^2}^4\,dr .
\end{split}
\end{equation}
By the Gronwall inequality \cite{heywood1990finite}, we have
\begin{equation}\label{2-18}
\begin{split}
\Vert \nabla \bm{u}_s \Vert_{L^2}^2
+\int_0^t
\|\mathcal{A}\bm{u}_s(r)\|_{L^2}^2 \ dr\leq C_{\nu}\left(1+\int_0^t \Vert \nabla p_{s-1}(r) \Vert_{L^2}^2 \ dr\right).
\end{split}
\end{equation}
Note that $\int _0^t \Vert \nabla p_s(r)\Vert_{L^2}^2 \ dr \leq C\int_0^t (\|\mathcal{A}\bm{u}_s(r)\|_{L^2}^2+\Vert \nabla p_{s-1}(r)\Vert_{L^2}^2) \ dr$. It follows that
\begin{equation}\label{2-19}
\begin{split}
\Vert \nabla\bm{u}_s(t)\Vert_{L^2}^2
+\int_0^{t}
\left(\Vert \Delta\bm{u}_s(r)\Vert_{L^2}^2+\Vert \nabla p_s(r)\Vert_{L^2}^2\right) \ dr
\leq C_{L23b}, \qquad 0\leq t\leq T.
\end{split}
\end{equation}
Thus, the high-order stability estimate is global-in-time in two dimensions.

For the three-dimensional case ($d=3$), the treatment of the nonlinear term in (\ref{2-16}) is no longer applicable. Therefore, a different estimate is required, which leads to a local-in-time high-order stability result. Note that
\begin{equation}\label{2-20}
\begin{split}
|B(\bm{u}_s,\bm{u}_s,\mathcal{A}\bm{u}_s)|
&\leq
C\|\nabla\bm{u}_s\|_{L^2}\|\nabla\bm{u}_s\|_{L^2}^{1/2}
\|\bm{u}_s\|_{H^2}^{1/2}
\|\mathcal{A}\bm{u}_s\|_{L^2}\\
&\leq
\delta\|\mathcal{A}\bm{u}_s\|_{L^2}^2
+\delta\|\Delta \bm{u}_s\|_{L^2}^2+
\|\nabla\bm{u}_s\|_{L^2}^6.
\end{split}
\end{equation}

Proceeding as in the two-dimensional case, and using the nonlinear Gronwall-Bihari inequality \cite{bihari1956generalization}, there exists a $T^*>0$ such that
\begin{equation}\label{2-21}
\|\nabla\bm{u}_s(t)\|_{L^2}^2
+
\int_0^{T^*}
\left(\|\mathcal{A}\bm{u}_s(r)\|_{L^2}^2 +\|\nabla p_s(r)\|_{L^2}^2\right) \ dr
\leq
C_{L23b}.
\end{equation}

This completes the proof.
\end{proof}

We are now in a position to establish the convergence of the SR-AV formulation to the original Navier-Stokes equations. Let
\[
\bm w_s=\bm u_s-\bm u,\qquad h_s=p_s-p,\qquad r_s=q_s-1. \]

Since $\nabla\cdot\bm u=0$, subtracting the Navier-Stokes equations (\ref{the1-1}) from the SR-AV system (\ref{the2-1}) yields
\begin{subequations}\label{2-22}
\begin{align}
&\frac{\partial \bm w_s}{\partial t}
-\nu\Delta\bm w_s+\nabla h_s
+\mathcal{N}(\bm w_s)=0,
\label{2-22a}\\
&\frac{d}{dt}r_s
=-\int_\Omega p_{s-1}\nabla\cdot\bm w_s\,d\bm x,
\label{2-22b}\\
&\frac{\partial(\nabla\cdot\bm w_s)}{\partial t}
+\alpha\nabla\cdot\bm w_s
=-\epsilon\left(
h_s-h_{s-1}-r_sp_{s-1}
\right),
\label{2-22c}
\end{align}
\end{subequations}
where
\[
\mathcal{N}(\bm w_s)
=(\bm u_s\cdot\nabla)\bm w_s
+\frac{1}{2}(\nabla\cdot\bm u_s)\bm w_s+(\bm{w}_s\cdot\nabla)\bm{u}+\frac{1}{2}(\nabla\cdot\bm w_s)\bm u.
\]
The initial conditions satisfy $\bm w_s(0)=0$ and $r_s(0)=0$.

The additional term $r_sp_{s-1}$ in (\ref{2-22c}) is caused by the auxiliary variable modification. As will be shown below, this term does not reduce the convergence order of the sequential regularization procedure.

\begin{theorem}
\label{theorem2.4}
Let $(\bm{u}, p)$ and $(\bm{u}_s, p_s, q_s)$ be the solutions corresponding to (\ref{the1-1}) and (\ref{the2-1}), respectively. Assume that $\int_0^{T_d} \Vert p_0\Vert_{H^1}^2 dt\leq 2+2\int_0^{T_d} \Vert p\Vert_{H^1}^2dt$ (e.g. $p_0=0$). Then there exists a positive constant $M$, independent of $\epsilon$ and the regularization level $s$, such that, for $6M \epsilon^2(1+2\int_0^{T_d}(\Vert p\Vert_{H^1}^2+\Vert p_0\Vert_{H^1}^2)dt)\leq 1$ and $s=1,2,\ldots$,
\begin{equation}\label{2-23}
\begin{aligned}
\|\bm{w}_s\|_{L^\infty(0,T_d;H^1)}^2
+\int_0^{T_d}
\Big(
\|\bm{w}_s\|_{H^2}^2
+\|(\bm{w}_s)_t\|_{L^2}^2
+\|h_s\|_{H^1}^2
\Big)\,dt\leq
\left(6M\epsilon^2\right)^s
\int_0^{T_d}
\|p_0-p\|_{H^1}^2\,dt ,
\end{aligned}
\end{equation}
where
\[
T_d=
\begin{cases}
T, & d=2,\\
T^*, & d=3,
\end{cases}
\]
Hence, the estimate is global in time for $d=2$, whereas it holds locally in time on $[0,T^*]$ for $d=3$.
\end{theorem}

\begin{proof}
The proof follows the strategy used in \cite{lu2008analysis}, based on mathematical induction and a bootstrap argument using the stability estimates established above. The main difference from the conventional
SRM analysis is the presence of the additional term $r_s p_{s-1}$ in (\ref{2-22c}), which arises from the auxiliary variable formulation.

The following claims are proved by induction and immediately imply the desired convergence estimate. Claim: For $s=1,2,\ldots$, we have
\begin{equation}\label{2-24}
\begin{split}
6M\epsilon^2\int_0^{T_d} \Vert h_{s-1}\Vert_{H^1}^2dt\leq 1,
\end{split}
\end{equation}
\begin{equation}\label{2-25}
\begin{split}
\|\bm{w}_s\|_{L^\infty(0,T_d;H^1)}^2
+\int_0^{T_d}
\Big(
\|\bm{w}_s\|_{H^2}^2
+\|(\bm{w}_s)_t\|_{L^2}^2
+\|h_s\|_{H^1}^2
\Big)\,dt\leq
 6M\epsilon^2
\int_0^{T_d}
\|h_{s-1}\|_{H^1}^2\,dt.
\end{split}
\end{equation}

For the case $s=1$, since our initial guess $p_0$ satisfies $\int_0^{T_d} \Vert p_0\Vert_{H^1}^2 dt\leq 2+2\int_0^{T_d} \Vert p\Vert_{H^1}^2dt$, it follows from (\ref{2-18}) that
\begin{equation}\label{2-26}
\begin{split}
\Vert \nabla\bm{u}_1\Vert_{L^2}^2
+\int_0^{t} \Vert \Delta\bm{u}_1\Vert_{L^2}^2 \ dt \leq M_1,
\end{split}
\end{equation}
where $M_1$ is independent of $\epsilon$ and the regularization level $s$. Therefore, by using Lemma 2.4 in \cite{lu2008analysis} to the error system (\ref{2-22a}) and (\ref{2-22c}), we obtain
\begin{equation}\label{2-27}
\begin{aligned}
&\|\bm{w}_1\|_{L^\infty(0,T_d;H^1)}^2
+\int_0^{T_d}
\left(
\|\bm{w}_1\|_{H^2}^2
+\|(\bm{w}_1)_t\|_{L^2}^2
+\|h_1\|_{H^1}^2
\right)\,dt\leq
M\epsilon^2\int_0^{T_d}\|h_1-h_{0}-r_1p_{0}\|_{H^1}^2\,dt
\\
&\leq
3M\epsilon^2
\int_0^{T_d}
\left(
\|h_1\|_{H^1}^2
+\|h_0\|_{H^1}^2
+|r_1|^2\|p_0\|_{H^1}^2
\right)\,dt,
\end{aligned}
\end{equation}
Moreover, it follows from the definition of $r_s$ that
\begin{equation}\label{2-28}
\begin{split}
|r_s|^2&\leq \left(\int_0^{T_d}\Vert \nabla p_{s-1}\Vert_{L^2}^2 dt\right) \left(\int_0^{T_d}\Vert  \bm{w}_{s}\Vert_{H^1}^2 dt\right)\\
&\leq CT_d\Vert \bm{w}_{s}\Vert_{L^\infty(0,T_d;H^1)}^2.
\end{split}
\end{equation}
We choose $\epsilon$ sufficiently small such that $3M\epsilon^2(1+CT_d)\leq \frac{1}{2}$, (\ref{2-27}) gives
\begin{equation}\label{2-29}
\begin{split}
&\|\bm{w}_1\|_{L^\infty(0,T_d;H^1)}^2
+\int_0^{T_d}
\left(
\|\bm{w}_1\|_{H^2}^2
+\|(\bm{w}_1)_t\|_{L^2}^2
+\|h_1\|_{H^1}^2
\right)\,dt\leq
6M\epsilon^2\int_0^{T_d}\|h_{0}\|_{H^1}^2 dt.
\end{split}
\end{equation}
Note that the claim (\ref{2-24}) for $s=1$ follows directly from the assumption 
\[
6M\epsilon^2
\left(
1+2\int_0^{T_d}
\left(
\|p\|_{H^1}^2+\|p_0\|_{H^1}^2
\right)\,dt
\right)
\leq 1.
\]
Thus, both (\ref{2-24}) and (\ref{2-25}) hold for $s=1$.

Suppose now that (\ref{2-24}) and (\ref{2-25}) hold up to the regularization level $s=k$. We show that they also hold for $s=k+1$.

Claims (\ref{2-24}) and (\ref{2-25}) for $s=k$ imply
\begin{equation}\label{2-30}
\begin{split}
&\int_0^{T_d}\Vert p_k\Vert_{H^1}^2 dt\leq 2\int_0^{T_d}\left(\Vert p\Vert_{H^1}^2+\Vert h_k\Vert_{H^1}^2\right) dt\\
&\leq 2\int_0^{T_d} \Vert p\Vert_{H^1}^2 dt+12M\epsilon^2\int_0^{T_d} \Vert h_{k-1}\Vert_{H^1}^2 dt\\
&\leq 2\int_0^{T_d} \Vert p\Vert_{H^1}^2 dt+2.
\end{split}
\end{equation}
The proof of claim (\ref{2-25}) for $s=k+1$ is similar to the case of $s=1$. Hence, we obtain
\begin{equation}\label{2-31}
\begin{aligned}
&\|\bm{w}_{k+1}\|_{L^\infty(0,T_d;H^1)}^2
+\int_0^{T_d}
\left(
\|\bm{w}_{k+1}\|_{H^2}^2
+\|(\bm{w}_{k+1})_t\|_{L^2}^2
+\|h_{k+1}\|_{H^1}^2
\right) dt\leq
6M\epsilon^2
\int_0^{T_d} \|h_{k}\|_{H^1}^2 dt.
\end{aligned}
\end{equation}
Therefore, the claim (\ref{2-25}) holds for $s=k+1$. Moreover, by (\ref{2-25}) at the level $s=k$, we have
\begin{equation}\label{2-32}
\begin{split}
6M\epsilon^2\int_0^{T_d} \Vert h_{k}\Vert_{H^1}^2dt\leq 6M\epsilon^2\left(6M\epsilon^2\int_0^{T_d} \|h_{k-1}\|_{H^1}^2 dt\right)\leq 6M\epsilon^2\leq 1.
\end{split}
\end{equation}
Thus, the claim (\ref{2-24}) also holds for $s=k+1$.
\end{proof}

\begin{remark}
Estimate (\ref{2-23}) shows that the SR-AV formulation preserves the successive convergence property of the conventional sequential regularization method. For each fixed regularization level $s$, the velocity and pressure errors are of order $O(\epsilon^s)$ in the corresponding norms. The estimate holds globally in
time in two dimensions, while in three dimensions it is valid on the local interval $[0,T^*]$, consistently with the high-order stability estimate established above.
\end{remark}

\section{The second-order SR-AV scheme and its stability}
\label{section3}
In this section, we first construct a second-order accurate SR-AV scheme for system (\ref{the2-1}), and then present stability results for the numerical scheme.

\subsection{Numerical scheme and its implementation}
Let $N>0$ be a positive integer and $T$ be the final time of computation. We set $\Delta t=T/N$ to be the uniform time step, and $\bm{u}^n_s$ and $p^n_s$ to be the numerical approximation of the fluid velocity $\bm{u}_s$ and pressure $p_s$ at $t=n\Delta t$ $(n\leq N)$, respectively. 

Based on the second-order Crank–Nicolson approximation, we propose the following midpoint scheme:
\begin{subequations}\label{3-1}
\begin{align}
&\frac{\bm{u}_s^{n+1}-\bm{u}_s^{n}}{\Delta t}
+\left(\bm{u}_s^{*,n+\frac{1}{2}}\cdot\nabla\right)
\bm{u}_s^{n+\frac{1}{2}}
+\frac{1}{2}
\left(\nabla\cdot\bm{u}_s^{*,n+\frac{1}{2}}\right)
\bm{u}_s^{n+\frac{1}{2}}
-\nu\Delta\bm{u}_s^{n+\frac{1}{2}}
+\nabla p_s^{n+\frac{1}{2}}
=\bm{f}^{n+\frac{1}{2}},
\label{3-1a}
\\
&\frac{q_s^{n+1}-q_s^{n}}{\Delta t}
=
-\int_{\Omega}
p_{s-1}^{n+\frac{1}{2}}
\nabla\cdot\bm{u}_s^{n+\frac{1}{2}}
\,d\bm{x},
\label{3-1b}
\\
&\frac{\nabla\cdot\bm{u}_s^{n+1}
-\nabla\cdot\bm{u}_s^{n}}{\Delta t}
+\alpha\nabla\cdot\bm{u}_s^{n+\frac{1}{2}}
=
-\epsilon
\left(
p_s^{n+\frac{1}{2}}
-q_s^{n+\frac{1}{2}}
p_{s-1}^{n+\frac{1}{2}}
\right),
\label{3-1c}
\end{align}
\end{subequations}
where $\bm{u}_s^{*,n+\frac{1}{2}}=\frac{3}{2}\bm{u}_s^{n}-\frac{1}{2}\bm{u}_s^{n-1}$. For $n=0$, we use the following first-order scheme to compute $\bm{u}_s^{*,1/2}$:
\begin{subequations}\label{3-2}
\begin{align}
&\frac{\bm{u}_s^{*,1/2}-\bm{u}_s^{0}}{\Delta t/2}
+\left(\bm{u}_s^{0}\cdot\nabla\right)\bm{u}_s^{*,1/2}
+\frac12
\left(\nabla\cdot\bm{u}_s^{0}\right)\bm{u}_s^{*,1/2}
-\nu\Delta\bm{u}_s^{*,1/2}
+\nabla p_s^{1/2}
=
\bm{f}^{1/2},
\label{3-2a}
\\
&\frac{q_s^{1/2}-q_s^{0}}{\Delta t/2}
=
-\int_{\Omega}
p_{s-1}^{1/2}
\nabla\cdot\bm{u}_s^{*,1/2}
\,d\bm{x},
\label{3-2b}\\
&\frac{\nabla\cdot\bm{u}_s^{*,1/2}
-\nabla\cdot\bm{u}_s^{0}}{\Delta t/2}
+\alpha\nabla\cdot\bm{u}_s^{*,1/2}
=
-\epsilon
\left(
p_s^{1/2}-q_s^{1/2}p_{s-1}^{1/2}
\right).
\label{3-2c}
\end{align}
\end{subequations}

We can see from (\ref{3-1}) that the velocity $\bm{u}_s^{n+1}$ and the auxiliary variable $q_s^{n+1}$ are coupled. To efficiently solve this coupled system, we employ a simple splitting technique to decouple the velocity from the auxiliary variable. Specifically, we decompose the velocity as $\bm{u}_s^{n+1}=\bm{u}_{s,1}^{n+1}+q_s^{n+1}\bm{u}_{s,2}^{n+1}$, and combine (\ref{3-1a}) and (\ref{3-1c}) to obtain the following two independent linear subproblems:
\begin{equation}\label{3-3}
\begin{split}
&\bm{u}_{s,1}^{n+1}
+\frac{\Delta t}{2}
\left(
\bm{u}_s^{*,n+\frac{1}{2}}\cdot\nabla
\right)
\bm{u}_{s,1}^{n+1}
+\frac{\Delta t}{4}
\left(
\nabla\cdot\bm{u}_s^{*,n+\frac{1}{2}}
\right)
\bm{u}_{s,1}^{n+1}
-\frac{\nu\Delta t}{2}\Delta\bm{u}_{s,1}^{n+1}
-\left(
\frac{1}{\epsilon}
+\frac{\alpha\Delta t}{2\epsilon}
\right)
\nabla\nabla\cdot\bm{u}_{s,1}^{n+1}
\\
&=
\Delta t\bm{f}^{n+\frac{1}{2}}+\bm{u}_s^n
-\frac{\Delta t}{2}
\left(
\bm{u}_s^{*,n+\frac{1}{2}}\cdot\nabla
\right)
\bm{u}_s^n
-\frac{\Delta t}{4}
\left(
\nabla\cdot\bm{u}_s^{*,n+\frac{1}{2}}
\right)
\bm{u}_s^n
\\
&
+\frac{\nu\Delta t}{2}\Delta\bm{u}_s^n
-\left(
\frac{1}{\epsilon}
-\frac{\alpha\Delta t}{2\epsilon}
\right)
\nabla\nabla\cdot\bm{u}_s^n
-\frac{\Delta t}{2}q_s^n
\nabla p_{s-1}^{n+\frac{1}{2}}, \quad \bm{u}_{s,1}^{n+1}|_{\partial \Omega}=0,
\end{split}
\end{equation}
and
\begin{equation}\label{3-4}
\begin{split}
&\bm{u}_{s,2}^{n+1}
+\frac{\Delta t}{2}
\left(
\bm{u}_s^{*,n+\frac{1}{2}}\cdot\nabla
\right)
\bm{u}_{s,2}^{n+1}
+\frac{\Delta t}{4}
\left(
\nabla\cdot\bm{u}_s^{*,n+\frac{1}{2}}
\right)
\bm{u}_{s,2}^{n+1}-\frac{\nu\Delta t}{2}\Delta\bm{u}_{s,2}^{n+1}
-\frac{2+\alpha\Delta t}{2\epsilon}
\nabla\nabla\cdot\bm{u}_{s,2}^{n+1}\\
&=
-\frac{\Delta t}{2}
\nabla p_{s-1}^{n+\frac{1}{2}}, \quad \bm{u}_{s,2}^{n+1}|_{\partial \Omega}=0.
\end{split}
\end{equation}
After solving (\ref{3-3}) and (\ref{3-4}) for $\bm{u}_{s,1}^{n+1}$ and $\bm{u}_{s,2}^{n+1}$, respectively, the auxiliary variable $q_s^{n+1}$ is determined from the following scalar algebraic equation:
\begin{equation}\label{3-5}
q_s^{n+1}
=
\frac{
q_s^n
-\displaystyle\frac{\Delta t}{2}
\int_{\Omega}
p_{s-1}^{n+\frac{1}{2}}
\nabla\cdot
\left(
\bm{u}_{s,1}^{n+1}
+\bm{u}_s^n
\right)
\,d\bm{x}
}{
1+\displaystyle\frac{\Delta t}{2}
\int_{\Omega}
p_{s-1}^{n+\frac{1}{2}}
\nabla\cdot
\bm{u}_{s,2}^{n+1}
\,d\bm{x}
}.
\end{equation}
With $q_s^{n+1}$ determined from (\ref{3-5}), the velocity is then recovered as
\begin{equation}\label{3-6}
\bm{u}_s^{n+1}
=
\bm{u}_{s,1}^{n+1}
+
q_s^{n+1}\bm{u}_{s,2}^{n+1},
\end{equation}
and the pressure $p_s^{n+\frac{1}{2}}$ is subsequently recovered from (\ref{3-1c}).

\begin{remark}
The above splitting procedure preserves the main computational advantages of the SRM while providing an efficient implementation of the auxiliary variable formulation. More precisely, at each time step, a total of $2s$ linear velocity  subproblems need to be solved for an $s$-stage SRM, with $s=2$ or $3$ typically being sufficient in practical computations. For each regularization stage, the two subproblems (\ref{3-3}) and (\ref{3-4}) involve the same linear operator and therefore share the same coefficient matrix after spatial discretization. The velocity and pressure are fully decoupled, and no artificial pressure boundary condition is required, thereby avoiding possible loss of accuracy caused by artificial pressure boundary treatments.
\end{remark}

\subsection{Stability results}
We first provide a low-order stability result for the SR-AV scheme (\ref{3-1}).
\begin{theorem}
\label{theorem3.2}
Assume that $\nabla\cdot\bm{u}_0=0$ and $\bm{f}\in L^\infty(0,T;L^2(\Omega))$. For the second-order time-discrete SR-AV scheme (\ref{3-1}), there exists a constant $C_{T32}$, independent of the viscosity $\nu$ and the regularization parameter $\epsilon$, such that
\begin{equation}\label{3-7}
\begin{split}
\Vert \bm{u}_s^{n+1}\Vert_{L^2}^2+\frac{1}{\epsilon}\Vert \nabla\cdot\bm{u}_s^{n+1}\Vert_{L^2}^2+|q_s^{n+1}|^2
+2\Delta t\sum_{k=0}^{n}
\left(
\nu\|\nabla\bm{u}_s^{k+\frac12}\|_{L^2}^2
+
\frac{\alpha}{\epsilon}
\|\nabla\cdot\bm{u}_s^{k+\frac12}\|_{L^2}^2
\right)\leq C_{T32}.
\end{split}
\end{equation}
In particular, if $f=0$, the scheme (\ref{3-1}) is unconditionally energy stable in the sense that the following discrete energy law holds:
\begin{equation}\label{3-8}
\begin{split}
\mathcal{E}(\bm{u}^{n+1}_s, q^{n+1}_s)-\mathcal{E}(\bm{u}^{n}_s, q_s^{n})=-\Delta t\nu\Vert\nabla\bm{u}_s^{n+\frac{1}{2}}\Vert_{L^2}^2-\frac{\Delta t\alpha}{\epsilon}\Vert\nabla\cdot \bm{u}_s^{n+\frac{1}{2}}\Vert_{L^2}^2\leq 0,
\end{split}
\end{equation}
where the discrete energy $\mathcal{E}$ is defined as $\mathcal{E}(\bm{u}_s^{n}, q_s^{n})=\frac{1}{2}\Vert\bm{u}_s^{n}\Vert_{L^2}^2+\frac{1}{2\epsilon}\Vert\nabla\cdot \bm{u}_s^{n}\Vert_{L^2}^2+\frac{(q_s^{n})^2}{2}-\frac{1}{2}$.
\end{theorem}

\begin{proof}
Taking the $L^2$ inner product of (\ref{3-1a}) with $\bm{u}_s^{n+\frac12}$ and using the skew-symmetry property
\[
B\left(
\bm{u}_s^{*,n+\frac12},
\bm{u}_s^{n+\frac12},
\bm{u}_s^{n+\frac12}
\right)=0,
\]
we obtain
\begin{equation}\label{3-9}
\begin{aligned}
\frac{1}{2\Delta t}
\left(
\|\bm{u}_s^{n+1}\|_{L^2}^2
-
\|\bm{u}_s^n\|_{L^2}^2
\right)
+
\nu
\|\nabla\bm{u}_s^{n+\frac12}\|_{L^2}^2
-
\left(
p_s^{n+\frac12},
\nabla\cdot\bm{u}_s^{n+\frac12}
\right)
=
\left(
\bm{f}^{n+\frac12},
\bm{u}_s^{n+\frac12}
\right).
\end{aligned}
\end{equation}

Taking the $L^2$ inner product of (\ref{3-1c}) with $\epsilon^{-1}\nabla\cdot\bm{u}_s^{n+\frac12}$ gives
\begin{equation}\label{3-10}
\begin{aligned}
&-\left(
p_s^{n+\frac12},
\nabla\cdot\bm{u}_s^{n+\frac12}
\right)\\
&=\frac{1}{2\epsilon\Delta t}
\left(
\|\nabla\cdot\bm{u}_s^{n+1}\|_{L^2}^2
-
\|\nabla\cdot\bm{u}_s^n\|_{L^2}^2
\right)
+\frac{\alpha}{\epsilon}
\|\nabla\cdot\bm{u}_s^{n+\frac12}\|_{L^2}^2-q_s^{n+\frac12}
\left(
p_{s-1}^{n+\frac12},
\nabla\cdot\bm{u}_s^{n+\frac12}
\right).
\end{aligned}
\end{equation}

On the other hand, multiplying (\ref{3-1b}) by $q_s^{n+\frac12}$ yields
\begin{equation}\label{3-11}
\frac{1}{2\Delta t}
\left(
|q_s^{n+1}|^2-|q_s^n|^2
\right)
=
-
q_s^{n+\frac12}
\left(
p_{s-1}^{n+\frac12},
\nabla\cdot\bm{u}_s^{n+\frac12}
\right).
\end{equation}

Combining (\ref{3-9})-(\ref{3-11}), we arrive at
\begin{equation}\label{3-12}
\begin{aligned}
E_s^{n+1}-E_s^n+2\Delta t\nu\|\nabla\bm{u}_s^{n+\frac12}\|_{L^2}^2+\frac{2\Delta t\alpha}{\epsilon}
\|\nabla\cdot\bm{u}_s^{n+\frac12}\|_{L^2}^2=2\Delta t\left(\bm{f}^{n+\frac12},
\bm{u}_s^{n+\frac12}
\right),
\end{aligned}
\end{equation}
where $E_s^n=\|\bm{u}_s^n\|_{L^2}^2+|q_s^n|^2+\frac{1}{\epsilon}\|\nabla\cdot\bm{u}_s^n\|_{L^2}^2$. If $\bm{f}=0$, (\ref{3-12}) immediately implies (\ref{3-8}), which proves the discrete energy dissipation law.

Dropping the nonnegative dissipation terms in (\ref{3-12}) gives
\begin{equation}\label{3-13}
\begin{split}
E_s^{n+1}-E_s^n\leq \Delta t \|\bm{f}^{n+\frac12}\|_{L^2}\left(\|\bm{u}_s^{n+1}\|_{L^2}+\|\bm{u}_s^{n}\|_{L^2}\right)\leq \Delta t \|\bm{f}^{n+\frac12}\|_{L^2}\left((E_s^{n+1})^{1/2}+(E_s^{n})^{1/2}\right).
\end{split}
\end{equation}
It follows that $(E_s^{n+1})^{1/2}-(E_s^{n})^{1/2}\leq \Delta t C_{\bm{f}}$. Summing over $k=0,\ldots,n$ yields
\begin{equation}\label{3-14}
\begin{split}
E_s^{n+1}\leq C_{\bm{f},T}.
\end{split}
\end{equation}
Finally, summing (\ref{3-12}) from $k=0$ to $n$, we can obtain the desired estimate (\ref{3-7}).

This completes the proof.
\end{proof}

We next present a high-order stability result, which is global in two-dimensions and local in three-dimensions.  

\begin{theorem}
\label{theorem3.3}
For $d=2$, the second-order SR-AV scheme (\ref{3-1}) satisfies the following global boundedness estimate:
\begin{equation}\label{3-15}
\begin{split}
\Vert \nabla \bm{u}_s^{n+1}\Vert_{L^2}^2+\Delta t\sum\limits_{k=0}^{n}\left(\Vert \Delta \bm{u}_s^{k+1}\Vert_{L^2}^2 + \Vert \nabla p_s^{k+1}\Vert_{L^2}^2\right)\leq C_{T33}, \quad n+1\leq \frac{T}{\Delta t},
\end{split}
\end{equation}
where $C_{T33}>0$ generally depends on the viscosity $\nu$.

For $d=3$, the above boundedness estimate holds locally in time. Specifically, there exists $T^*>0$ such that (\ref{3-15}) holds for $0<T\leq T^*$.
\end{theorem}

\begin{proof}
We introduce the notation
\[
D_t\bm{u}_s^{n+1}
=
\frac{\bm{u}_s^{n+1}-\bm{u}_s^n}{\Delta t}
\]
and define the discrete operator
\begin{equation}\label{3-16}
\begin{aligned}
\mathcal{A}_s^{n+\frac12}
:={}&
-\nu\Delta\bm{u}_s^{n+\frac12}
-\frac{1}{\epsilon}
\nabla D_t
\left(\nabla\cdot\bm{u}_s^{n+1}\right)
-\frac{\alpha}{\epsilon}
\nabla\nabla\cdot\bm{u}_s^{n+\frac12}.
\end{aligned}
\end{equation}

We first consider the 2-D case. It follows from (\ref{3-1a}), (\ref{3-1c}) and (\ref{3-16}) that
\begin{equation}\label{3-17}
\begin{aligned}
D_t\bm{u}_s^{n+1}
+
\mathcal{B}
\left(
\bm{u}_s^{*,n+\frac12},
\bm{u}_s^{n+\frac12}
\right)
+
\mathcal{A}_s^{n+\frac12}
+
q_s^{n+\frac12}
\nabla p_{s-1}^{n+\frac12}
=
\bm{f}^{n+\frac12},
\end{aligned}
\end{equation}

Taking the $L^2$ inner product of (\ref{3-17}) with $2\Delta t\mathcal{A}_s^{n+\frac12}$, we obtain
\begin{equation}\label{3-18}
\begin{aligned}
&
\nu
\left(
\|\nabla\bm{u}_s^{n+1}\|_{L^2}^2
-
\|\nabla\bm{u}_s^n\|_{L^2}^2
\right)+
\frac{\alpha}{\epsilon}
\left(
\|\nabla\cdot\bm{u}_s^{n+1}\|_{L^2}^2
-
\|\nabla\cdot\bm{u}_s^n\|_{L^2}^2
\right)
+
\frac{2\Delta t}{\epsilon}
\left\|
D_t(\nabla\cdot\bm{u}_s^{n+1})
\right\|_{L^2}^2
+
2\Delta t
\|\mathcal{A}_s^{n+\frac12}\|_{L^2}^2
\\
&=
-2\Delta t
B\left(
\bm{u}_s^{*,n+\frac12},
\bm{u}_s^{n+\frac12},
\mathcal{A}_s^{n+\frac12}
\right)
-2\Delta t q_s^{n+\frac12}
\left(
\nabla p_{s-1}^{n+\frac12},
\mathcal{A}_s^{n+\frac12}
\right)
+
2\Delta t
\left(
\bm{f}^{n+\frac12},
\mathcal{A}_s^{n+\frac12}
\right).
\end{aligned}
\end{equation}

By (\ref{the2-3}) and the low-order stability estimate established in Theorem \ref{theorem3.2}, we derive
\begin{equation}\label{3-19}
\begin{aligned}
2\left|
q_s^{n+\frac12}
\left(
\nabla p_{s-1}^{n+\frac12},
\mathcal{A}_s^{n+\frac12}
\right)
\right|
\leq
\frac{1}{3}
\|\mathcal{A}_s^{n+\frac12}\|_{L^2}^2
+
C
\|\nabla p_{s-1}^{n+\frac12}\|_{L^2}^2,
\end{aligned}
\end{equation}
\begin{equation}\label{3-20}
2\left|
\left(
\bm{f}^{n+\frac12},
\mathcal{A}_s^{n+\frac12}
\right)
\right|
\leq
\frac{1}{3}
\|\mathcal{A}_s^{n+\frac12}\|_{L^2}^2
+
C
\|\bm{f}^{n+\frac12}\|_{L^2}^2.
\end{equation}
and
\begin{equation}\label{3-21}
\begin{aligned}
&\left|
B\left(
\bm{u}_s^{*,n+\frac12},
\bm{u}_s^{n+\frac12},
\mathcal{A}_s^{n+\frac12}
\right)
\right|\\
&\leq
C
\|\bm{u}_s^{*,n+\frac12}\|_{L^2}^{\frac{1}{2}}\|\nabla\bm{u}_s^{*,n+\frac12}\|_{L^2}^{\frac{1}{2}}
\|\nabla\bm{u}_s^{n+\frac12}\|_{L^2}^{\frac12}
\|\Delta\bm{u}_s^{n+\frac12}\|_{L^2}^{\frac12}
\|\mathcal{A}_s^{n+\frac12}\|_{L^2}\\
&+C\|\nabla\bm{u}_s^{*,n+\frac12}\|_{L^2}
\|\bm{u}_s^{n+\frac12}\|_{L^2}^{\frac12}
\|\Delta\bm{u}_s^{n+\frac12}\|_{L^2}^{\frac12}
\|\mathcal{A}_s^{n+\frac12}\|_{L^2}
\\
&\leq
\frac{1}{3}
\|\mathcal{A}_s^{n+\frac12}\|_{L^2}^2
+
\delta
\|\Delta\bm{u}_s^{n+\frac12}\|_{L^2}^2
+
C
\|\nabla\bm{u}_s^{*,n+\frac12}\|_{L^2}^2
\left(\|\nabla\bm{u}_s^{n+\frac12}\|_{L^2}^2+\|\nabla\bm{u}_s^{*,n+\frac12}\|_{L^2}^2\right).
\end{aligned}
\end{equation}

Substituting (\ref{3-19})-(\ref{3-21}) into (\ref{3-18}) and summing over $k=0,\ldots,n$, we have
\begin{equation}\label{3-22}
\begin{aligned}
&\|\nabla\bm{u}_s^{n+1}\|_{L^2}^2
+\frac{\alpha}{\epsilon}
\|\nabla\cdot\bm{u}_s^{n+1}\|_{L^2}^2
+\Delta t\sum\limits_{k=0}^{n}\frac{1}{\epsilon}
\left\|
D_t(\nabla\cdot\bm{u}_s^{k+1})
\right\|_{L^2}^2
+\Delta t\sum\limits_{k=0}^{n}
\|\mathcal{A}_s^{k+\frac12}\|_{L^2}^2
\\
&\leq \delta\Delta t\sum\limits_{k=0}^{n}\|\Delta\bm{u}_s^{k+\frac12}\|_{L^2}^2+
C\Delta t\sum_{k=0}^{n}
\left(\|\nabla\bm{u}_s^{*,k+\frac12}\|_{L^2}^2+\|\nabla\bm{u}_s^{k+\frac12}\|_{L^2}^2\right)\|\nabla\bm{u}_s^{*,k+\frac12}\|_{L^2}^2\\
&+
C\Delta t\sum_{k=0}^{n}
\left(
\|\nabla p_{s-1}^{k+\frac12}\|_{L^2}^2
+
\|\bm{f}^{k+\frac12}\|_{L^2}^2\right).
\end{aligned}
\end{equation}

By using the regularity estimate $\Delta t\sum_{k=0}^{n}\|\Delta\bm{u}_s^{k+1}\|_{L^2}^2\leq C\Delta t\sum_{k=0}^{n}\|\mathcal{A}_s^{k+1}\|_{L^2}^2$ in Lemma 3.1 of \cite{lu2008analysis} and the Gronwall inequality, we obtain
\begin{equation}\label{3-23}
\begin{aligned}
\|\nabla\bm{u}_s^{n+1}\|_{L^2}^2+\Delta t\sum\limits_{k=0}^{n}\|\mathcal{A}_s^{k+\frac12}\|_{L^2}^2\leq C
\left(
1+
\Delta t\sum_{k=0}^{n}
\|\nabla p_{s-1}^{k+\frac12}\|_{L^2}^2
\right).
\end{aligned}
\end{equation}
Moreover, it follows from (\ref{3-1c}) that
\begin{equation}\label{3-24}
\begin{aligned}
\Delta t\sum\limits_{k=0}^{n} \|\nabla p_s^{k+\frac12}\|_{L^2}^2&\leq C\Delta t\sum\limits_{k=0}^{n}\left(\|\mathcal{A}_s^{k+\frac12}\|_{L^2}^2+\nu^2\|\Delta\bm{u}_s^{k+\frac12}\|_{L^2}^2+|q_s^{k+\frac12}|^2
\|\nabla p_{s-1}^{k+\frac12}\|_{L^2}^2\right)\\
&\leq C\left(1+\Delta t\sum_{k=0}^{n}\|\nabla p_{s-1}^{k+\frac12}\|_{L^2}^2\right).
\end{aligned}
\end{equation}

Combining (\ref{3-23}) and (\ref{3-24}), and noting that the regularization level $s$ is finite and the initial pressure approximation $p_0$ is prescribed, a recursive argument with respect to $s$ yields the desired estimate (\ref{3-15}).

We next consider the 3-D case. The main difference lies in the estimate of the nonlinear term. By the Gagliardo–Nirenberg interpolation inequality,
\begin{equation}\label{3-25}
\begin{aligned}
&
\left|
B\left(
\bm{u}_s^{*,n+\frac12},
\bm{u}_s^{n+\frac12},
\mathcal{A}_s^{n+\frac12}
\right)
\right|\leq
C
\|\nabla\bm{u}_s^{*,n+\frac12}\|_{L^2}
\|\nabla\bm{u}_s^{n+\frac12}\|_{L^2}^{\frac12}
\|\Delta\bm{u}_s^{n+\frac12}\|_{L^2}^{\frac12}
\|\mathcal{A}_s^{n+\frac12}\|_{L^2}
\\
&\leq
\frac{1}{3}
\|\mathcal{A}_s^{n+\frac12}\|_{L^2}^2
+
\delta
\|\Delta\bm{u}_s^{n+\frac12}\|_{L^2}^2+
C
\|\nabla\bm{u}_s^{*,n+\frac12}\|_{L^2}^4
\|\nabla\bm{u}_s^{n+\frac12}\|_{L^2}^2\\
&\leq \frac{1}{3}
\|\mathcal{A}_s^{n+\frac12}\|_{L^2}^2
+
\delta
\|\Delta\bm{u}_s^{n+\frac12}\|_{L^2}^2+
C
\|\nabla\bm{u}_s^{*,n+\frac12}\|_{L^2}^4\|\bm{u}_s^{n+\frac12}\|_{L^2}\|\Delta \bm{u}_s^{n+\frac12}\|_{L^2}\\
&\leq \frac{1}{3}
\|\mathcal{A}_s^{n+\frac12}\|_{L^2}^2
+
\delta
\|\Delta\bm{u}_s^{n+\frac12}\|_{L^2}^2+
C
\|\nabla\bm{u}_s^{*,n+\frac12}\|_{L^2}^8.
\end{aligned}
\end{equation}

Following the same argument as in the two-dimensional case, we arrive at a nonlinear discrete inequality of the form
\begin{equation}\label{3-26}
Z_{n+1}
+
\Delta t\sum_{k=0}^{n}
\left(\|\mathcal{A}_s^{k+\frac12}\|_{L^2}^2+\|\nabla p_{s}^{k+\frac12}\|_{L^2}^2\right)
\leq
C\left(1+\Delta t\sum_{k=0}^{n}\|\nabla p_{s-1}^{k+\frac12}\|_{L^2}^2\right)
+
C\Delta t\sum_{k=0}^{n}G(Z_k),
\end{equation}
where $Z_{n}=\Vert\nabla \bm{u}_s^{n}\Vert_{L^2}^2$ and $G(Z)=Z^4$. By Lemma \ref{lemma2-1}, there exists $T^*>0$, independent of $\Delta t$, such that (\ref{3-15}) holds for $0<T\leq T^*$.

This completes the proof.
\end{proof}

\section{Convergence analysis}
\label{section4}
In this section, we establish error estimates for the SR-AV scheme (\ref{3-1}). We mainly focus on the two-dimensional case, for which global-in-time error estimates are derived. For the three-dimensional case, the error analysis follows a similar argument, but leads only to local-in-time estimates, consistent with the corresponding stability analysis.

We first state the regularity assumptions required for the second-order consistency analysis.

\begin{assumption}\label{assumption4-1}
For each fixed regularization level $s$, the solution $(\bm{u}_s,p_s,q_s)$ of the SR-AV system (\ref{the2-1}) satisfies the following regularity condition:
\begin{equation}\label{4-1}
\begin{aligned}
&\bm{u}_s\in L^\infty(0,T;W^{1,\infty})\cap W^{2,2}(0,T;H^1)\cap W^{3,2}(0,T;H^1),\quad p_s\in W^{1,\infty}(0,T; H^1)\cap W^{2,\infty}(0,T; H^1),\\
& q_s\in L^{\infty}(0,T)\cap W^{3,2}(0,T).
\end{aligned}
\end{equation}
\end{assumption}

We set
\begin{equation}\nonumber
\begin{aligned}
e_{\bm{u}_s}^{n}=\bm{u}_s^{n}-\bm{u}_s(t^{n}), \quad e_{p_s}^{n}=p_s^{n}-p_s(t^{n}), \quad e_{q_s}^{n}=q_s^{n}-q_s(t^{n}).
\end{aligned}
\end{equation}

The truncation form of the system (\ref{the2-1}) is as follows:
\begin{subequations}\label{4-2}
	\begin{align}
&\frac{\bm{u}_s(t^{n+1})-\bm{u}_s(t^{n})}{\Delta t}+\left((\bm{u}_s(t^{*,n+\frac{1}{2}}) \cdot \nabla) \bm{u}_s(t^{n+\frac{1}{2}})+\frac{1}{2}(\nabla\cdot\bm{u}_s(t^{*,n+\frac{1}{2}}))\bm{u}_s(t^{n+\frac{1}{2}})\right)-\nu \Delta \bm{u}_s(t^{n+\frac{1}{2}}) \label{4-2a}\\
&+\nabla p_s(t^{n+\frac{1}{2}}) =\bm{f}(t^{n+\frac{1}{2}})+R_{\bm{u}}^{n+\frac{1}{2}}, \nonumber\\
&\frac{q_s(t^{n+1})-q_s(t^{n})}{\Delta t} = -\int_\Omega p_{s-1}(t^{n+\frac{1}{2}})\left(\nabla\cdot \bm{u}_s(t^{n+\frac{1}{2}}) \right)\ d\bm{x}+R_{q}^{n+\frac{1}{2}},\label{4-2b}\\
&\frac{\nabla\cdot\bm{u}_s(t^{n+1})-\nabla\cdot\bm{u}_s(t^{n}) }{\Delta t}+\alpha\nabla\cdot \bm{u}_s(t^{n+\frac{1}{2}})=-\epsilon \left(p_s(t^{n+\frac{1}{2}})-q_s(t^{n+\frac{1}{2}})p_{s-1}(t^{n+\frac{1}{2}})\right)+R_{d}^{n+\frac{1}{2}}, \label{4-2c}
\end{align}
\end{subequations}
where $\bm{u}_s(t^{*,n+\frac{1}{2}})=\frac{3}{2}\bm{u}_s(t^n)-\frac{1}{2}\bm{u}_s(t^{n-1})$. Furthermore, the truncation errors can be expressed as follows:
\begin{equation}\label{4-3}
\begin{aligned}
&R_{\bm{u}}^{n+\frac{1}{2}}=\frac{1}{2\Delta t}\left(\int_{t^n}^{t^{n+\frac{1}{2}}} (t^n-r)^2\frac{\partial^3\bm{u}_s}{\partial r^3}dr+\int_{t^{n+\frac{1}{2}}}^{t^{n+1}} (t^{n+1}-r)^2\frac{\partial^3\bm{u}_s}{\partial r^3}dr\right)\\
&+\left(\frac12\int_{t^{n-1}}^{t^n}(t^{n-1}-r)\frac{\partial^2\bm{u}_s}{\partial r^2}\,dr
-\int_{t^n}^{t^{n+\frac12}}(t^{n+\frac12}-r)\frac{\partial^2\bm{u}_s}{\partial r^2}\,dr\right)\cdot\nabla\bm{u}_s(t^{n+\frac12})\\
&+\frac{1}{2}\left(\frac12\int_{t^{n-1}}^{t^n}(t^{n-1}-r)\frac{\partial^2(\nabla\cdot\bm{u}_s)}{\partial r^2}\,dr
-\int_{t^n}^{t^{n+\frac12}}(t^{n+\frac12}-r)\frac{\partial^2(\nabla\cdot\bm{u}_s)}{\partial r^2}\,dr\right)\bm{u}_s(t^{n+\frac12}),
\end{aligned}
\end{equation}
\begin{equation}\label{4-4}
\begin{aligned}
R_{q}^{n+\frac{1}{2}}=\frac{1}{2\Delta t}\left(\int_{t^n}^{t^{n+\frac{1}{2}}} (t^n-t)^2\frac{d^3q_s}{d t^3}dt+\int_{t^{n+\frac{1}{2}}}^{t^{n+1}} (t^{n+1}-t)^2\frac{d^3 q_s}{d t^3}dt\right),
\end{aligned}
\end{equation}
and 
\begin{equation}\label{4-5}
\begin{aligned}
R_{d}^{n+\frac{1}{2}}=\frac{1}{2\Delta t}\left(\int_{t^n}^{t^{n+\frac{1}{2}}} (t^n-r)^2\frac{\partial^3(\nabla\cdot\bm{u}_s)}{\partial r^3}dr+\int_{t^{n+\frac{1}{2}}}^{t^{n+1}} (t^{n+1}-r)^2\frac{\partial^3(\nabla\cdot\bm{u}_s)}{\partial r^3}dr\right).
\end{aligned}
\end{equation}

We first establish a low-order error estimate for the velocity, with a convergence constant independent of the viscosity $\nu$.

\begin{theorem}
\label{theorem4.2}
Assume that the regularity condition \eqref{4-1} holds. Then, for any fixed finite regularization level $s\geq 1$, there exists a positive constant $C_{T42}$ such that, for any $0\leq n\leq N-1$,
\begin{equation}\label{4-6}
\begin{aligned}
\|e_{\bm u_s}^{n+1}\|_{H^1}^2+|e_{q_s}^{n+1}|^2+\Delta t\sum_{k=0}^n\Vert e_{p_{s}}^{k+\frac12}\Vert_{H^1}^2+\Delta t\sum_{k=0}^n\| e_{\bm u_s}^{k+\frac{1}{2}}\|_{H^2}^2\leq C_{T42}(\Delta t)^4.
\end{aligned}
\end{equation}
Here, $C_{T42}$ is independent of $\Delta t$ but may depend on $\nu$ and $s$.
\end{theorem}

\begin{proof}
Combining (\ref{3-1}) and (\ref{4-2}), we obtain the error equations for $\bm{u}_s$, $p_s$ and $q_s$:
\begin{subequations}\label{4-7}
\begin{align}
&\frac{e_{\bm u_s}^{n+1}-e_{\bm u_s}^{n}}{\Delta t}+J_1-\nu\Delta e_{\bm u_s}^{n+\frac12}+\nabla e_{p_s}^{n+\frac12}+R_{\bm{u}}^{n+\frac{1}{2}}=0, \label{4-7a}\\
&\frac{e_{q_s}^{n+1}-e_{q_s}^{n}}{\Delta t}=-\left(p_{s-1}^{n+\frac{1}{2}}\nabla\cdot e_{\bm{u}_s}^{n+\frac{1}{2}}+e_{p_{s-1}}^{n+\frac{1}{2}}\nabla\cdot \bm{u}_s(t^{n+\frac{1}{2}})\right)-R_{q}^{n+\frac{1}{2}}, \label{4-7b}\\
&\frac{\nabla\cdot e_{\bm u_s}^{n+1}-\nabla\cdot e_{\bm u_s}^{n}}{\Delta t}+\alpha \nabla\cdot e_{\bm u_s}^{n+\frac{1}{2}}=-\epsilon\left(e_{p_{s}}^{n+\frac{1}{2}}-\left(e_{q_s}^{n+\frac{1}{2}}p_{s-1}^{n+\frac{1}{2}}+q_s(t^{n+\frac{1}{2}})e_{p_{s-1}}^{n+\frac{1}{2}})\right)\right)-R_{d}^{n+\frac{1}{2}}. \label{4-7c}
\end{align}
\end{subequations}
Here
\begin{equation}\label{4-8}
\begin{aligned}
J_1=\bm{u}_s^{*,n+\frac{1}{2}}\cdot\nabla e_{\bm{u}_s}^{n+\frac{1}{2}}+e_{\bm{u}_s}^{*,n+\frac{1}{2}}\cdot \nabla \bm{u}_s(t^{n+\frac{1}{2}})+\frac{1}{2}\left(\nabla\cdot \bm{u}_s^{*,n+\frac{1}{2}}\right)e_{\bm{u}_s}^{n+\frac{1}{2}}+\frac{1}{2}(\nabla\cdot e_{\bm{u}_s}^{*,n+\frac{1}{2}})\bm{u}_s(t^{n+\frac{1}{2}}).
\end{aligned}
\end{equation}

Substituting (\ref{4-7c}) into (\ref{4-7a}) to eliminate the pressure $e_{p_s}^{n+\frac{1}{2}}$, we obtain
\begin{equation}\label{4-9}
\begin{aligned}
&\frac{e_{\bm u_s}^{n+1}-e_{\bm u_s}^{n}}{\Delta t}+J_1-\nu\Delta e_{\bm u_s}^{n+\frac12}-\frac{\nabla\nabla\cdot e_{\bm u_s}^{n+1}-\nabla\nabla\cdot e_{\bm u_s}^{n}}{\epsilon\Delta t}-\frac{\alpha}{\epsilon}\nabla\nabla\cdot e_{\bm u_s}^{n+\frac12}\\
&+e_{q_s}^{n+\frac{1}{2}}\nabla p_{s-1}^{n+\frac{1}{2}}+q_s(t^{n+\frac{1}{2}})\nabla e_{p_{s-1}}^{n+\frac{1}{2}}+R_{\bm{u}}^{n+\frac{1}{2}}-\frac{1}{\epsilon}\nabla R_{d}^{n+\frac{1}{2}}=0.
\end{aligned}
\end{equation}

Taking the inner products of (\ref{4-9}) and (\ref{4-7b}) with $2\Delta t e_{\bm u_s}^{n+\frac12}$ and $2\Delta t e_{q_s}^{n+\frac{1}{2}}$, respectively, and then adding the resulting identities, we obtain
\begin{equation}\label{4-10}
\begin{aligned}
&\left(\Vert e_{\bm u_s}^{n+1}\Vert_{L^2}^2-\Vert e_{\bm u_s}^{n}\Vert_{L^2}^2\right)+\frac{1}{\epsilon}\left(\Vert\nabla\cdot e_{\bm u_s}^{n+1}\Vert_{L^2}^2-\Vert\nabla\cdot e_{\bm u_s}^{n}\Vert_{L^2}^2\right)+\left(|e_{q_s}^{n+1}|^2-|e_{q_s}^{n}|^2\right)\\
&+2\nu\Delta t\Vert\nabla e_{\bm u_s}^{n+\frac12}\Vert_{L^2}^2+\frac{2\alpha\Delta t}{\epsilon}\Vert\nabla \cdot e_{\bm u_s}^{n+\frac12}\Vert_{L^2}^2\\
&=-2\Delta t\left(J_1, e_{\bm u_s}^{n+\frac12}\right)-2\Delta t(q_s(t^{n+\frac{1}{2}})\nabla e_{p_{s-1}}^{n+\frac{1}{2}}, e_{\bm u_s}^{n+\frac12})-2\Delta t(e_{p_{s-1}}^{n+\frac{1}{2}}e_{q_s}^{n+\frac{1}{2}}, \nabla\cdot \bm{u}_s(t^{n+\frac{1}{2}}))\\
&-2\Delta t(R_{\bm{u}}^{n+\frac{1}{2}}, e_{\bm u_s}^{n+\frac12})+\frac{2\Delta t}{\epsilon}(\nabla R_{d}^{n+\frac{1}{2}}, e_{\bm u_s}^{n+\frac12})-2\Delta t(R_q^{n+\frac{1}{2}}, e_{q_s}^{n+\frac{1}{2}}).
\end{aligned}
\end{equation}

By the Cauchy-Schwarz inequality, the skew-symmetry property of the trilinear form, and Assumption \ref{assumption4-1}, we estimate
\begin{equation}\label{4-11}
\begin{aligned}
&|(J_1, e_{\bm u_s}^{n+\frac12})|\leq C_\delta\left(\Vert e_{\bm{u}_s}^{*,n+\frac{1}{2}}\Vert_{L^2}^2+\Vert \nabla\cdot e_{\bm{u}_s}^{*,n+\frac{1}{2}}\Vert_{L^2}^2\right)+\delta\Vert e_{\bm u_s}^{n+\frac12}\Vert_{L^2}^2,\\
&\left|(q_s(t^{n+\frac{1}{2}})\nabla e_{p_{s-1}}^{n+\frac{1}{2}}, e_{\bm u_s}^{n+\frac12})\right|\leq C_\delta\Vert \nabla e_{p_{s-1}}^{n+\frac{1}{2}}\Vert_{L^2}^2+\delta\Vert e_{\bm u_s}^{n+\frac12}\Vert_{L^2}^2,\\
&\left|(e_{p_{s-1}}^{n+\frac{1}{2}}e_{q_s}^{n+\frac{1}{2}}, \nabla\cdot \bm{u}_s(t^{n+\frac{1}{2}}))\right|\leq C_\delta\Vert e_{p_{s-1}}^{n+\frac{1}{2}}\Vert_{L^2}^2+\delta \left|e_{q_s}^{n+\frac{1}{2}}\right|^2.
\end{aligned}
\end{equation}

For the truncation terms, we have the following estimates:
\begin{equation}\label{4-12}
\begin{aligned}
&2\Delta t\left|(R_{\bm{u}}^{n+\frac{1}{2}}, e_{\bm u_s}^{n+\frac12})\right|\leq C\Delta t\Vert R_{\bm{u}}^{n+\frac{1}{2}}\Vert_{L^2}^2+\delta \Delta t\Vert e_{\bm u_s}^{n+\frac12}\Vert_{L^2}^2,\\
&\frac{2\Delta t}{\epsilon}|(\nabla R_{d}^{n+\frac{1}{2}}, e_{\bm u_s}^{n+\frac12})|\leq C\frac{\Delta t}{\epsilon}\Vert R_d^{n+\frac{1}{2}}\Vert_{L^2}^2+\frac{2\alpha\Delta t}{\epsilon}\Vert \nabla\cdot e_{\bm u_s}^{n+\frac12}\Vert_{L^2}^2,\\
&2\Delta t\left|(R_q^{n+\frac{1}{2}}, e_{q_s}^{n+\frac{1}{2}})\right|\leq C\Delta t \Vert R_q^{n+\frac{1}{2}}\Vert_{L^2}^2+\delta\Delta t \left|e_{q_s}^{n+\frac{1}{2}}\right|^2.
\end{aligned}
\end{equation}

Note that
\begin{equation}\nonumber
\begin{aligned}
\nabla\cdot \bm{u}_s(t)=-\epsilon e^{-\alpha t}\int_0^t (p_s-q_sp_{s-1}) e^{\alpha r} dr.
\end{aligned}
\end{equation}
Under the regularity assumption (\ref{4-1}), the quantity $p_s-q_sp_{s-1}$ and its temporal derivatives up to second order are uniformly bounded with respect to $\epsilon$. Therefore, $\Vert \frac{\partial^3 (\nabla\cdot \bm{u}_s)}{\partial t^3}\Vert_{L^2}^2\leq C\epsilon^2$.

Combining (\ref{4-10})-(\ref{4-12}), summing over $k=0, ..., n$, we obtain
\begin{equation}\label{4-13}
\begin{aligned}
&\Vert e_{\bm u_s}^{n+1}\Vert_{L^2}^2+\frac{1}{\epsilon}\Vert\nabla\cdot e_{\bm u_s}^{n+1}\Vert_{L^2}^2+|e_{q_s}^{n+1}|^2+\nu\Delta t\sum_{k=0}^n\Vert\nabla e_{\bm u_s}^{k+\frac12}\Vert_{L^2}^2+\frac{\alpha\Delta t}{\epsilon}\sum_{k=0}^n\Vert\nabla \cdot e_{\bm u_s}^{k+\frac12}\Vert_{L^2}^2\\
&\leq C\Delta t\sum_{k=0}^n\Vert e_{p_{s-1}}^{k+\frac{1}{2}}\Vert_{H^1}^2+C\Delta t\sum_{k=1}^n\left(\Vert e_{\bm{u}_s}^{k}\Vert_{L^2}^2+\Vert e_{\bm{u}_s}^{k-1}\Vert_{L^2}^2+\Vert \nabla\cdot e_{\bm{u}_s}^{k}\Vert_{L^2}^2+\Vert \nabla\cdot e_{\bm{u}_s}^{k-1}\Vert_{L^2}^2+|e_{q_s}^k|^2\right)\\
&+C\left(\Vert \bm{u}_s\Vert_{W^{3,2}(0,T;H^1)}^2+\Vert \bm{u}_s\Vert_{W^{2,2}(0,T;H^1)}^2+\left|q_s\right|_{W^{3,2}(0,T)}^2+\epsilon\right)(\Delta t)^4.
\end{aligned}
\end{equation}

Applying the Gronwall inequality technique it is easy to obtain
\begin{equation}\label{4-14}
\begin{aligned}
\Vert e_{\bm u_s}^{n+1}\Vert_{L^2}^2+\frac{1}{\epsilon}\Vert \nabla\cdot e_{\bm u_s}^{n+1}\Vert_{L^2}^2+|e_{q_s}^{n+1}|^2+\nu\Delta t\sum_{k=0}^n\Vert\nabla e_{\bm u_s}^{k+\frac12}\Vert_{L^2}^2\leq C\left((\Delta t)^4+\Delta t\sum_{k=0}^n\Vert e_{p_{s-1}}^{k+\frac{1}{2}}\Vert_{H^1}^2\right).
\end{aligned}
\end{equation}

Let $\mathcal{A}e_{\bm{u}_s}^{n+\frac{1}{2}}=-\nu\Delta e_{\bm u_s}^{n+\frac12}-\frac{1}{\epsilon}\nabla D_t(\nabla\cdot e_{\bm u_s}^{n+1})-\frac{\alpha}{\epsilon}\nabla\nabla\cdot e_{\bm u_s}^{n+\frac12}$, where
$D_t(\nabla\cdot e_{\bm u_s}^{n+1})
=
\frac{
\nabla\cdot e_{\bm u_s}^{n+1}
-
\nabla\cdot e_{\bm u_s}^{n}}
{\Delta t}$.
By taking the inner product of (\ref{4-9}) with $2\Delta t\,\mathcal{A}e_{\bm u_s}^{n+\frac12}$, we obtain
\begin{equation}\label{4-15}
\begin{aligned}
&\nu
\left(
\|\nabla e_{\bm u_s}^{n+1}\|_{L^2}^2
-
\|\nabla e_{\bm u_s}^{n}\|_{L^2}^2
\right)
+
\frac{\alpha}{\epsilon}
\left(
\|\nabla\cdot e_{\bm u_s}^{n+1}\|_{L^2}^2
-
\|\nabla\cdot e_{\bm u_s}^{n}\|_{L^2}^2
\right)+\frac{2\Delta t}{\epsilon}
\left\|
D_t(\nabla\cdot e_{\bm u_s}^{n+1})
\right\|_{L^2}^2
+
2\Delta t
\|\mathcal{A}e_{\bm u_s}^{n+\frac12}\|_{L^2}^2
\\
&=
-2\Delta t
\left(
J_1,
\mathcal{A}e_{\bm u_s}^{n+\frac12}
\right)
-2\Delta t\,e_{q_s}^{n+\frac12}
\left(
\nabla p_{s-1}^{n+\frac12},
\mathcal{A}e_{\bm u_s}^{n+\frac12}
\right)
-2\Delta t\,q_s(t^{n+\frac12})
\left(
\nabla e_{p_{s-1}}^{n+\frac12},
\mathcal{A}e_{\bm u_s}^{n+\frac12}
\right)
\\
&\quad
-2\Delta t
\left(
R_{\bm u}^{n+\frac12},
\mathcal{A}e_{\bm u_s}^{n+\frac12}
\right)
+
\frac{2\Delta t}{\epsilon}
\left(
\nabla R_d^{n+\frac12},
\mathcal{A}e_{\bm u_s}^{n+\frac12}
\right).
\end{aligned}
\end{equation}

By (\ref{the2-3}) and Theorem \ref{theorem3.3}, we estimate
\begin{equation}\label{4-16}
\begin{aligned}
&|(J_1, \mathcal{A}e_{\bm u_s}^{n+\frac12})|\leq C\left(\Vert \bm{u}_s^{*,n+\frac{1}{2}}\Vert_{H^1}\Vert e_{\bm{u}_s}^{n+\frac{1}{2}}\Vert_{H^1}^{1/2}\Vert \Delta e_{\bm{u}_s}^{n+\frac{1}{2}}\Vert_{L^2}^{1/2}\Vert \mathcal{A}e_{\bm u_s}^{n+\frac12}\Vert_{L^2}+\Vert  e_{\bm{u}_s}^{*,n+\frac{1}{2}}\Vert_{H^1}\Vert \mathcal{A}e_{\bm u_s}^{n+\frac12}\Vert_{L^2} \right)\\
&\leq C\Vert e_{\bm{u}_s}^{n+\frac{1}{2}}\Vert_{L^2}^2+\delta\Vert \Delta e_{\bm{u}_s}^{n+\frac{1}{2}}\Vert_{L^2}^2+\delta\Vert \mathcal{A}e_{\bm u_s}^{n+\frac12}\Vert_{L^2}^2,\\
&|e_{q_s}^{n+\frac12}\left(\nabla p_{s-1}^{n+\frac12},\mathcal{A}e_{\bm u_s}^{n+\frac12}\right)|\leq C\Vert \nabla p_{s-1}^{n+\frac{1}{2}}\Vert_{L^2}^2|e_{q_s}^{n+\frac{1}{2}}|^2+\delta\Vert \mathcal{A}e_{\bm u_s}^{n+\frac12}\Vert_{L^2}^2,\\
&|\left(\nabla e_{p_{s-1}}^{n+\frac12}, \mathcal{A}e_{\bm u_s}^{n+\frac12}\right)|\leq C\Vert e_{p_{s-1}}^{n+\frac12}\Vert_{H^1}^2+\delta\Vert \mathcal{A}e_{\bm u_s}^{n+\frac12}\Vert_{L^2}^2,
\end{aligned}
\end{equation}
and 
\begin{equation}\label{4-17}
\begin{aligned}
&2\Delta t\left|(R_{\bm{u}}^{n+\frac{1}{2}}, \mathcal{A}e_{\bm u_s}^{n+\frac12})\right|\leq C\Delta t\Vert R_{\bm{u}}^{n+\frac{1}{2}}\Vert_{L^2}^2+\delta \Delta t\Vert \mathcal{A}e_{\bm u_s}^{n+\frac12}\Vert_{L^2}^2,\\
&\frac{2\Delta t}{\epsilon}|(\nabla R_{d}^{n+\frac{1}{2}}, \mathcal{A}e_{\bm u_s}^{n+\frac12})|\leq \frac{C\Delta t}{\epsilon^2}\Vert \nabla R_d^{n+\frac{1}{2}}\Vert_{L^2}^2+\delta \Delta t\Vert \mathcal{A}e_{\bm u_s}^{n+\frac12}\Vert_{L^2}^2.
\end{aligned}
\end{equation}
Here, we used the interpolation inequality $ \Vert e_{\bm{u}_s}^{n+\frac{1}{2}}\Vert_{H^1}\leq C\Vert e_{\bm{u}_s}^{n+\frac{1}{2}}\Vert_{L^2}^{1/2}\Vert \Delta e_{\bm{u}_s}^{n+\frac{1}{2}}\Vert_{L^2}$. Moreover, since
\begin{equation}\nonumber
\begin{aligned}
\nabla\nabla\cdot \bm{u}_s(t)
=
-\epsilon e^{-\alpha t}
\int_0^t
\left(
\nabla p_s(r)-q_s(r)\nabla p_{s-1}(r)
\right)e^{\alpha r}\,dr,
\end{aligned}
\end{equation}
it follows from Assumption \ref{assumption4-1} that $\|\frac{\partial^3(\nabla\nabla\cdot \bm{u}_s)}{\partial t^3}\|_{L^2}^2\leq C\epsilon^2$.

Combining (\ref{4-14})-(\ref{4-17}), and summing over $k=0, ..., n$, we obtain
\begin{equation}\label{4-18}
\begin{aligned}
&\|\nabla e_{\bm u_s}^{n+1}\|_{L^2}^2+\frac{\alpha}{\epsilon}\|\nabla \cdot e_{\bm u_s}^{n+1}\|_{L^2}^2+\frac{\Delta t}{\epsilon}\sum_{k=0}^n \Vert
D_t(\nabla\cdot e_{\bm u_s}^{k+1})\Vert_{L^2}^2+\Delta t\sum_{k=0}^n \Vert\mathcal{A}e_{\bm u_s}^{k+\frac12}\Vert_{L^2}^2\\
&\leq C(\Delta t)^4+C\Delta t\sum_{k=0}^{n} \Vert e_{\bm u_s}^{k+\frac{1}{2}}\Vert_{L^2}^2+C\Delta t\sum_{k=0}^{n}\Vert e_{p_{s-1}}^{k+\frac12}\Vert_{H^1}^2+\max_{1\le k\le n}|e_{q_s}^{k+\frac12}|^2
\Delta t\sum_{k=0}^{n}\|\nabla p_{s-1}^{k+\frac12}\|_{L^2}^2\\
&\leq C\left((\Delta t)^4+\Delta t\sum_{k=0}^{n}\Vert e_{p_{s-1}}^{k+\frac12}\Vert_{H^1}^2\right).
\end{aligned}
\end{equation}

On the other hand, it follows from (\ref{4-7c}) and Lemma 3.1 of \cite{lu2008analysis} that
\begin{equation}\label{4-19}
\begin{aligned}
&\Delta t\sum_{k=0}^n\Vert e_{p_{s}}^{k+\frac12}\Vert_{H^1}^2\\
&\leq C\left(\Delta t\sum_{k=0}^n \Vert \mathcal{A}e_{\bm u_s}^{k+\frac12}\Vert_{L^2}^2+\max_{1\le k\le n}|e_{q_s}^{k+\frac12}|^2
\Delta t\sum_{k=0}^{n}\|\nabla p_{s-1}^{k+\frac12}\|_{L^2}^2+\Delta t\sum_{k=0}^{n}\|e_{p_{s-1}}^{k+\frac12}\|_{H^1}^2+\frac{\Delta t}{\epsilon^2}\sum_{k=0}^{n}\Vert\nabla R_{d}^{k+\frac{1}{2}}\Vert_{L^2}^2\right)\\
&\leq C((\Delta t)^4+\Delta t\sum_{k=0}^{n}\Vert e_{p_{s-1}}^{k+\frac12}\Vert_{H^1}^2).
\end{aligned}
\end{equation}
Then it follows from (\ref{4-14}), (\ref{4-18}) and (\ref{4-19}) that 
\begin{equation}\label{4-20}
\begin{aligned}
\|e_{\bm u_s}^{n+1}\|_{H^1}^2+|e_{q_s}^{n+1}|^2+\Delta t\sum_{k=0}^n\Vert e_{p_{s}}^{k+\frac12}\Vert_{H^1}^2+\Delta t\sum_{k=0}^n\|\Delta e_{\bm u_s}^{k+\frac{1}{2}}\|_{L^2}^2\leq C(\Delta t)^4+C\Delta t\sum_{k=0}^{n}\Vert e_{p_{s-1}}^{k+\frac12}\Vert_{H^1}^2.
\end{aligned}
\end{equation}
Since the regularization level $s$ is finite, applying (\ref{4-20}) recursively with respect to $s$, starting from $e_{p_0}=0$, yields the desired estimate.

This completes the proof.
\end{proof}

\begin{remark}
Theorem \ref{theorem4.2} estimates the temporal discretization error between the numerical solution and the corresponding continuous SR-AV solution for a fixed regularization level $s$. On the other hand, Theorem \ref{theorem2.4} shows that the continuous SR-AV solution converges successively to the solution of the original Navier-Stokes equations as the regularization level increases. Therefore, the total error of the numerical approximation can be decomposed into the temporal discretization error and the regularization error. In particular, for each fixed finite $s$, the corresponding velocity error has the form
\[
\|\bm u_s^n-\bm u(t^n)\|_{H^1}^2+\Delta t\sum_{k=0}^n\Vert e_{\bm{u}_s}^{k+\frac{1}{2}}\Vert_{H^2}^2
\leq
C\left((\Delta t)^4+\epsilon^{2s}\right).
\]
\end{remark}

\begin{remark}[A viscosity-independent low-order error estimate]
\label{rem:srav-viscosity-low-order}
Fix a regularization level $s$ and parameters $\epsilon,\alpha>0$. Assume that the regularity bounds in Assumption~4.1, including the relevant bounds for the prescribed pressure input $p_{s-1}$, and the consistency and initialization estimates are uniform for $0<\nu\le\nu_0$. Then the low-order energy argument leading to (4.14) gives 
\[
\begin{aligned}
&\max_{0\le n\le N}
\left(
 \|e_{u_s}^{n}\|_{L^2}^{2}
 +|e_{q_s}^{n}|^{2}
 +\frac{1}{\epsilon}
  \|\nabla\cdot e_{u_s}^{n}\|_{L^2}^{2}
\right)\leq C\left(
 (\Delta t)^4
 +\Delta t\sum_{k=0}^{N-1}
  \|e_{p_{s-1}}^{k+1/2}\|_{H^1}^{2}
\right),
\end{aligned}
\]
where $C$ is independent of $\nu$ and $\Delta t$, but may depend on $s$, $\epsilon$, $\alpha$, $T$, $\nu_0$, and the uniform regularity bounds. In particular, if the same prescribed pressure input is used in the continuous and time-discrete problems and is supplied exactly at the discrete evaluation times, the pressure-data error vanishes. This yields second-order temporal convergence of the velocity in $\ell^\infty(L^2)$ with a viscosity-independent constant. For the full multilevel algorithm, a closed viscosity-independent bound follows provided that the pressure error from the preceding regularization level satisfies a second-order bound in $\ell^2(H^1)$ with a constant independent of $\nu$.
\end{remark}

\section{Numerical experiments}
\label{section5}
In this section, we first present two numerical examples to verify the accuracy and stability of the proposed SR-AV scheme. Moreover, we compare the numerical results with those obtained by the projection method and the conventional SRM. Subsequently, the performance of the proposed method is further assessed through a benchmark two-dimensional flow past a cylinder, where the computed hydrodynamic quantities are compared with reference data. Finally, a three-dimensional flow past an aircraft model is considered to demonstrate the capability of the SR-AV scheme for more complex incompressible flow problems. 

For the spatial discretization, we employ continuous piecewise quadratic finite elements for the velocity and continuous piecewise linear finite elements for the pressure. More precisely, on a conforming triangulation (or tetrahedralization in three dimensions) $\mathcal{T}_h$ of $\Omega$, we use the $\mathcal{P}_2$ finite element space for the velocity and the $\mathcal{P}_1$ finite element space for the pressure. Unless otherwise specified, we set the penalty parameter $\epsilon=10^{-4}$, the regularization level $s=2$, the damping parameter $\alpha=1$, and the initial pressure guess $p_0=0$.

\textbf{Example 1.} In this example, we set the domain $\Omega=[0,1]^2$, Reynolds number $Re=200$ ($\nu=0.005$), and $T=1$. The homogeneous Dirichlet boundary condition is imposed. The right hand side of the equations is computed according to the following analytic solution given by
\begin{equation}\label{the5-1}
\begin{split}
&u_1(t,x,y) = \sin(t)\sin^2(\pi x)\sin(2\pi y),\\
&u_2(t,x,y) = -\sin(t)\sin(2\pi x)\sin^2(\pi y),\\
&p(t,x,y) = \sin(t)\cos(\pi x)\sin(\pi y).
\end{split}
\end{equation}

We also use the following second-order linearized projection scheme for comparison with the SR-AV scheme.
\begin{equation}\label{the5-2}
\begin{split}
&\frac{\widetilde{\bm{u}}^{n+1}-\bm{u}^{n}}{\Delta t}+(\bm{u}^{*,n+\frac{1}{2}}\cdot \nabla)\widetilde{\bm{u}}^{n+\frac{1}{2}}-\nu\Delta \widetilde{\bm{u}}^{n+\frac{1}{2}}+\nabla p^n=\bm{f}^{n+\frac{1}{2}},\quad \widetilde{\bm{u}}^{n+1}|_{\partial\Omega}=0,\\
&\frac{\bm{u}^{n+1}-\widetilde{\bm{u}}^{n+1}}{\Delta t}+\frac{1}{2}\nabla\left(p^{n+1}-p^n\right)=0,\\
&\nabla\cdot \bm{u}^{n+1}=0, \quad \bm{u}^{n+1}\cdot \bm{n}|_{\partial\Omega}=0.
\end{split}
\end{equation}
The pressure is obtained by solving the Poisson equation $\Delta (p^{n+1}-p^n)=\frac{2}{\Delta t}\nabla\cdot \widetilde{\bm{u}}^{n+1}$ with the artificial Neumann boundary condition $\nabla (p^{n+1}-p^n)\cdot \bm{n}|_{\partial \Omega}=0$.

Numerical results are presented in Table \ref{table1}. It can be observed that the proposed SR-AV scheme achieves the expected second-order temporal accuracy for both the velocity and the pressure, which is consistent with the theoretical error estimate established in Section 4. Moreover, the SR-AV scheme provides better accuracy than the projection scheme for all the time steps considered. In particular, the velocity errors obtained by the SR-AV scheme are approximately one order of magnitude smaller than those of the projection scheme. This difference may be partly attributed to the artificial Neumann boundary condition imposed on the pressure in the projection method, which can introduce a numerical boundary layer and increase the error near the boundary \cite{guermond2006overview}. In contrast, the SR-AV scheme does not require any artificial pressure boundary condition and retains the original homogeneous Dirichlet boundary condition for the velocity. The results therefore demonstrate that the proposed SR-AV scheme not only achieves the desired second-order accuracy, but also yields smaller velocity and pressure errors than the projection method for the present test problem.

\begin{table}[H]
\caption{Errors and convergence rates for $Re=200$ with the second-order scheme using the SR-AV and the projection methods at $t=T=1$.}\label{table1}
\centering
\scalebox{0.95}{
\resizebox{\linewidth}{!}{
\begin{tabular}{c c c c c c c} \hline  		
\large $\Delta t$ (SR-AV/Projection)  & \large $\Vert e_{\bm{u}}\Vert_{L^2}$  & \large Order & \large $\Vert e_{\bm{u}}\Vert_{L^\infty}$ & \large Order & \large $\Vert e_{p}\Vert_{L^2}$ & \large Order \\ \hline
\large $1/4$		& \large 1.76e-3/2.27e-2 & - & \large 3.09e-3/6.94e-2 & - & \large 5.24e-3/1.28e-2 & \large - \\ [4pt]
\large $1/8$	& \large 4.47e-4/5.83e-3  & \large 1.98/\large 1.96 & \large 7.97e-4/1.65e-2 & \large 1.95/2.07  & \large 1.36e-3/3.31e-3 &  \large 1.95/1.95 \\  [4pt]
\large$1/16$	& \large 1.12e-4/1.46e-3  &  \large 1.99/\large 2.00 & \large 2.03e-4/4.42e-3 & \large 1.98/1.90  & \large 3.45e-4/8.67e-4 & \large 1.97/1.93 \\  [4pt]
\large$1/32$	& \large 2.82e-5/3.63e-4   & \large 2.00/\large 2.00 & \large 5.16e-5/1.03e-3 & \large 1.98/2.10    & \large 8.95e-5/2.27e-4 & \large 1.95/1.94 \\  \hline 
\end{tabular}}
	}
\end{table}

\textbf{Example 2.}
This example investigates the discrete energy behavior of the SR-AV scheme and the corresponding conventional SRM \cite{lin1997sequential}. We consider an unforced double shear layer flow in the periodic domain $\Omega=(0,2\pi)^2$. The initial velocity is prescribed by
\begin{equation}\label{eq:double_shear_initial}
\bm u_0(x,y)
=
\begin{pmatrix}
U(y)\\
\delta\sin x
\end{pmatrix},
\qquad
U(y)=
\begin{cases}
\tanh\!\left[
2\kappa\left(\dfrac{y}{2\pi}-\dfrac14\right)
\right],
&0\leq y\leq\pi,\\[2mm]
\tanh\!\left[
2\kappa\left(\dfrac34-\dfrac{y}{2\pi}\right)
\right],
&\pi<y<2\pi,
\end{cases}
\end{equation}
where $\kappa=15$ and $\delta=0.05$. This divergence-free initial field consists of two horizontal shear layers centered at $y=\pi/2$ and $y=3\pi/2$, together with a small transverse perturbation. The parameters are set to $Re=10^5 \ (\nu=10^{-5})$, $\epsilon=0.02$, $\Delta t=2.5$ and $T=50$. 

Figure \ref{figure1} presents the discrete energy evolution and the divergence error of the SR-AV scheme and the conventional SRM. As shown in Figure \ref{figure1}(a), the modified energy of the SR-AV scheme
maintains a monotone decreasing behavior throughout the simulation, which agrees with the discrete energy dissipation property established by the theoretical analysis. The original energy and the modified energy of the SR-AV scheme exhibit a similar decreasing trend, while their difference becomes more pronounced during the stage of rapid energy dissipation. This indicates that the auxiliary variable correction effectively modifies the discrete energy functional and provides the additional dissipation mechanism required for high Reynolds number flows. In comparison, as shown in the enlarged view of Figure \ref{figure1}(a), the energy of the conventional SRM shows small oscillations in the later stage of the simulation, indicating that the discrete energy dissipation property is no longer preserved.

Figure \ref{figure1}(b) further compares the divergence errors of the two schemes. Both methods maintain a relatively small divergence error during the early stage of the simulation. However, after the formation of complex flow structures, the divergence error produced by the conventional SRM becomes larger than that of the SR-AV scheme. In particular, the SR-AV scheme exhibits a smaller divergence error in the later stage, indicating that the proposed auxiliary variable formulation provides improved control of the incompressibility constraint.

\begin{figure}[H]
\centering
\subfigure[Discrete energy evolution]
{
    \includegraphics[width=0.45\textwidth]{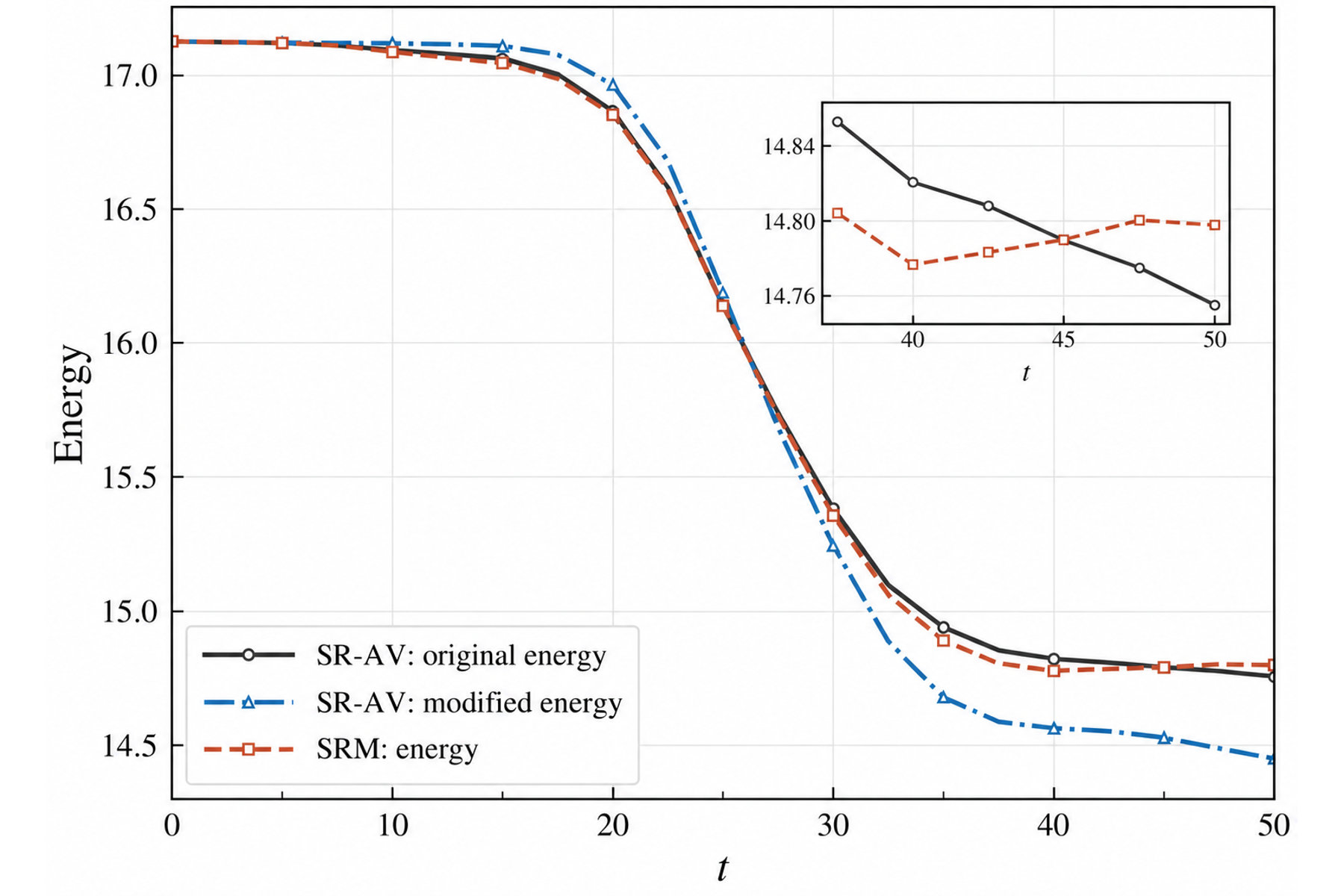}
}
\subfigure[Divergence error evolution]
{
    \includegraphics[width=0.46\textwidth]{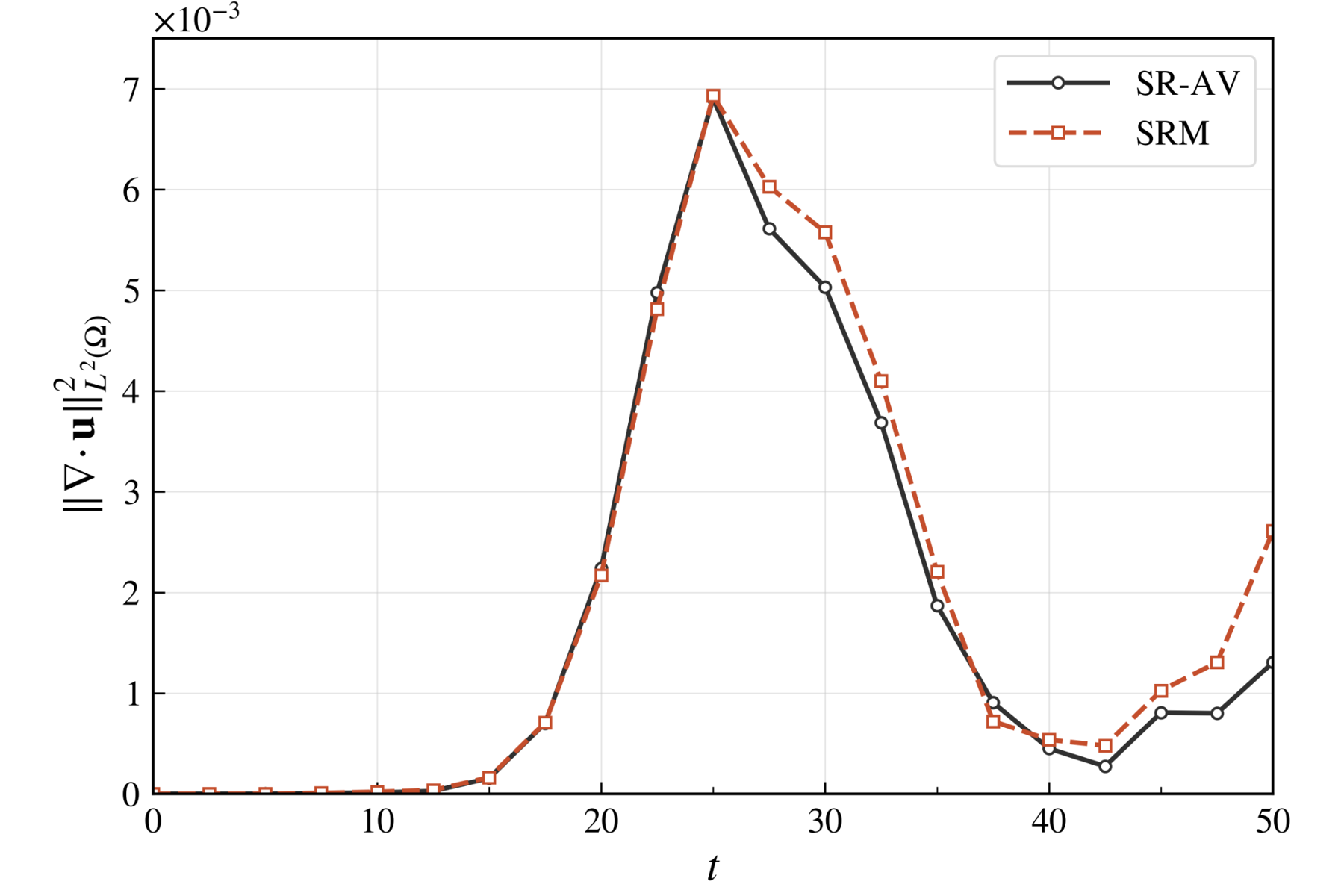}
}
\caption{Comparison of the SR-AV scheme and the conventional SRM for the double shear layer problem with $Re=10^5$.}
\label{figure1}
\end{figure}

\textbf{Example 3.} We consider the two-dimensional flow past a cylinder \cite{john2004reference}. The computational domain is a rectangle $\Omega=(0, 2.2)\times(0,0.41)$, excluding a disk (the cross section of the cylinder) inside, with the center of the disk located at $(0.2, 0.2)$ and a radius of $0.05$. We set $\nu=10^{-3}$, $T=8$ and $\Delta t=5\times 10^{-3}$. The velocity prescribed at both the inlet and the outlet is given by
\[
\boldsymbol{u}(t,0,y)=\boldsymbol{u}(t,2.2,y)
=\left(\frac{6y(0.41-y)}{0.41^2}\sin\frac{\pi t}{8},0\right),
\]
while the no-slip boundary condition is imposed on the channel walls and the cylinder surface.

The snapshots of the velocity field are shown in Figure \ref{figure2}. It can be observed that the flow gradually separates as time evolves. Complex vortex structures develop behind the cylinder, accompanied by pronounced vortex shedding and oscillatory flow behavior. Figure \ref{figure3} compares the time histories of $C_D$, $C_L$, and $\Delta p$ obtained by the SR-AV scheme with the benchmark results of John \cite{john2004reference}. We can see that the initial increase and subsequent decrease of the drag coefficient and the pressure difference are accurately captured. The rapid oscillations of the lift coefficient associated with the alternating vortex shedding are also well resolved. The lift coefficient is more sensitive to the spatial resolution and the geometric approximation of the circular boundary, which accounts for its slightly larger relative discrepancy. Therefore, the proposed SR-AV scheme successfully resolves the vortex-shedding mechanism and yields the temporal evolution of the drag coefficient, lift coefficient, and pressure difference in good agreement with the benchmark results. This further demonstrates the accuracy and effectiveness of the proposed scheme for complex unsteady incompressible flows.

\begin{figure}[H]
	\centering
	\includegraphics[scale=0.54]{example3velocity.png}
	\caption{Snapshots of the velocity field of flow passing a cylinder.}
	\label{figure2}
\end{figure}

\begin{figure}[H]
	\centering
	\includegraphics[scale=0.50]{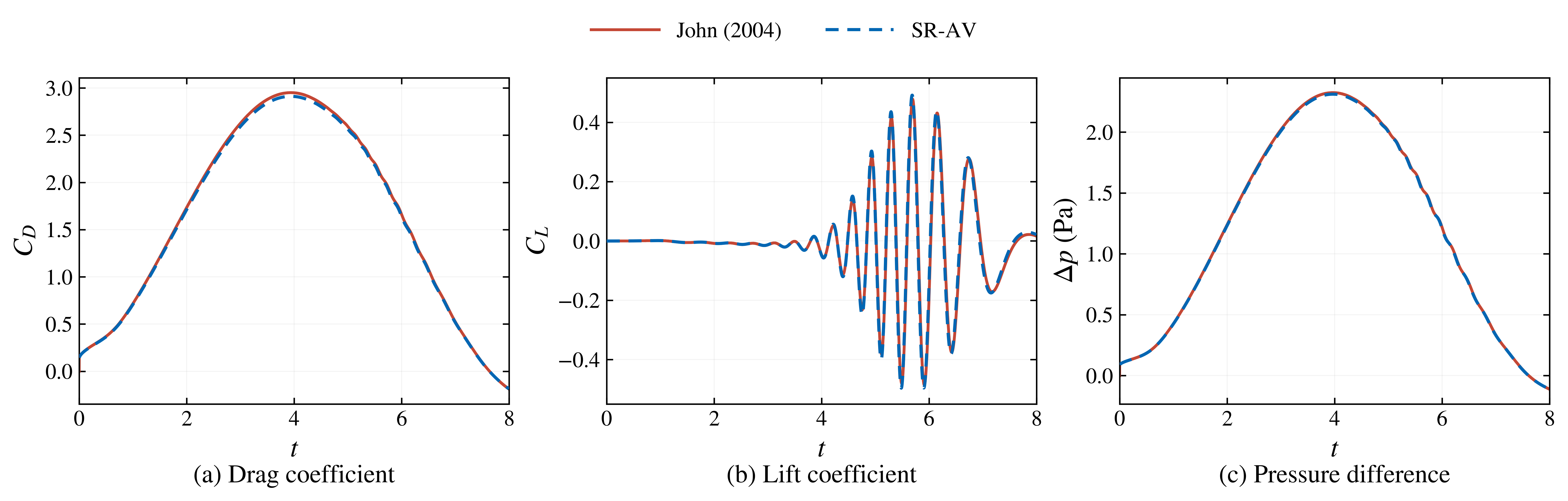}
	\caption{Comparison of the drag coefficient $C_D$, lift coefficient $C_L$, and pressure difference $\Delta p$ obtained by the proposed SR-AV scheme with the benchmark data of John (2004).}
	\label{figure3}
\end{figure}

\textbf{Example 4.} This example investigates 3-D incompressible flow past a simplified aircraft model based on the NASA Common Research Model (CRM) \footnote{\url{https://commonresearchmodel.larc.nasa.gov/geometry/}}. The computational domain is $\Omega=(0,3)\times(0,1.5)\times(0,1.5)$. We consider $Re=1000$ and $Re=10^4$, corresponding to $\nu=10^{-3}$ and $\nu=10^{-4}$, respectively. The final time and the time step are set to $T=3$ and $\Delta t=5\times10^{-3}$, and the external body force is $\bm f=\bm 0$. A constant inflow velocity $\bm u=(1,0,0)^{\top}$ is imposed at $x=0$, while the no-slip boundary condition is imposed on the four lateral walls $y=0,1.5$ and $z=0,1.5$. At the outlet $x=3$, the homogeneous Neumann boundary condition is imposed for the velocity.

Figure \ref{figure4} shows the vorticity-magnitude isosurfaces $|\bm\omega|=2$, where $\bm\omega=\nabla\times\bm u$. The upper and lower rows correspond to $Re=1000$ and $Re=10^4$, respectively. As time increases, the vorticity structures extend farther downstream. For $Re=1000$, the displayed surfaces remain comparatively smooth and continuous. For $Re=10^4$, they exhibit more pronounced deformation and finer irregularity, particularly at $t=3$. These results illustrate the application of the SR-AV scheme to a complex 3-D geometry and the Reynolds-number dependence of the computed transient wake.

\begin{figure}[H]
	\centering
	\includegraphics[scale=0.32]{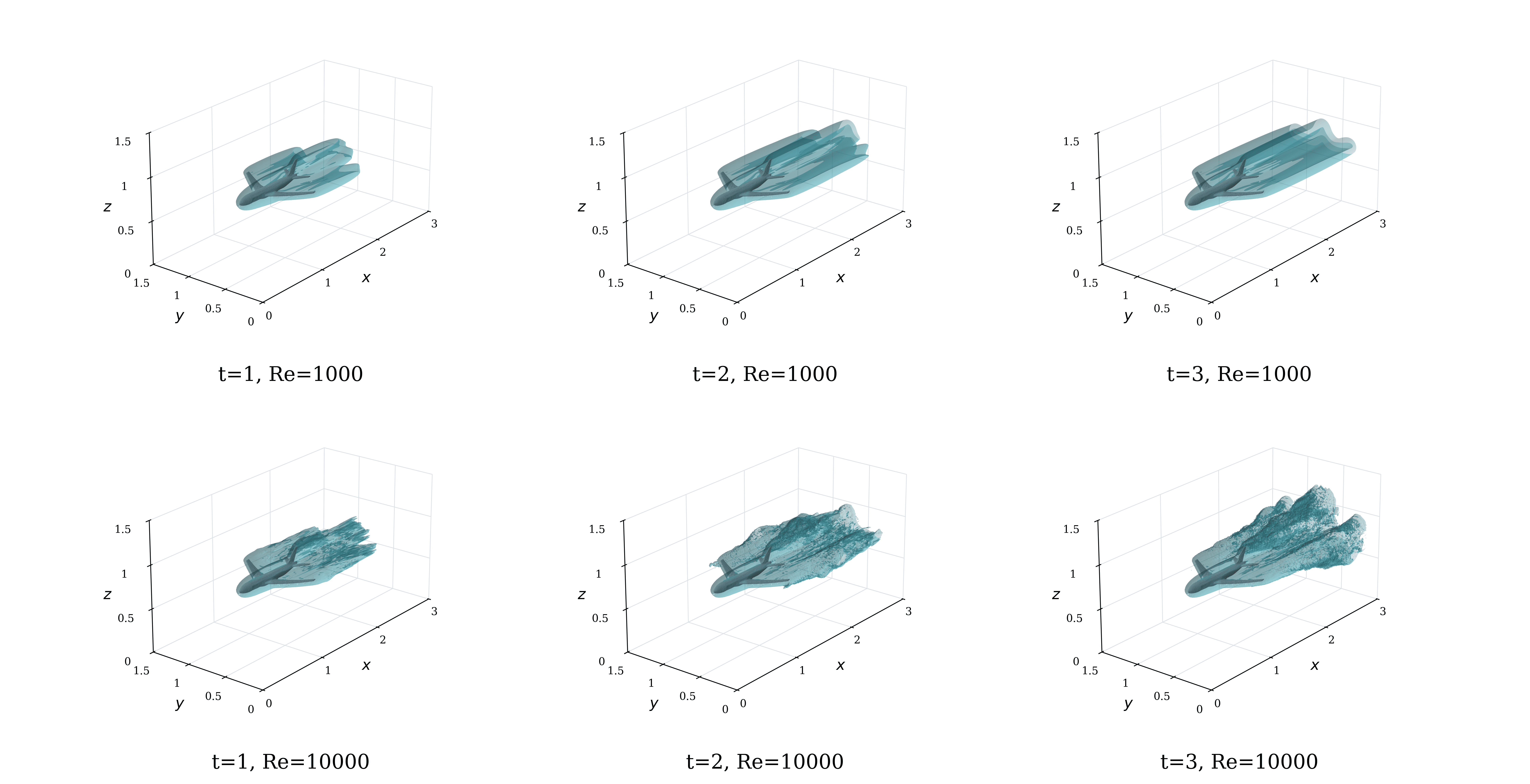}
	\caption{Comparison of 3-D vorticity-magnitude isosurfaces for flow past an aircraft at $\mathrm{Re}=1000$ (top) and $\mathrm{Re}=10^4$ (bottom).}
	\label{figure4}
\end{figure}

\section{Conclusions and remarks}
\label{section6}
This work proposes a sequential regularization auxiliary variable method for the incompressible Navier-Stokes equations. By introducing an auxiliary variable into the sequential regularization formulation, the proposed approach compensates for the coupling between the pressure from the preceding regularization level and the divergence defect of the current velocity, thereby recovering an energy-stable structure while retaining the boundary-preserving feature of the conventional sequential regularization method. At the continuous level, we establish stability estimates and show that the SR-AV formulation preserves the successive convergence property of sequential regularization, with the velocity and pressure errors decreasing as $O(\epsilon^s)$ at the $s$-th regularization level. A second-order temporal discretization is then developed and shown to be unconditionally energy stable. The stability and convergence properties of the SR-AV scheme are rigorously analyzed.

We believe that the ideas presented in this work can be applied to construct unconditionally stable and convergent schemes for fluid-coupled multiphysics systems, such as the Cahn-Hilliard-Navier-Stokes \cite{chen2024convergence} and magnetohydrodynamic \cite{wang2022optimal} models. The error analysis (both continuous and discrete levels) of this method for these systems and comparisons with other numerical methods will be the focus of our future work.

\section*{Acknowledgments}
Z. Wang is partially supported by the National Natural Science Foundation of China (Grant No. 12501535) and the Scientific Research Fund of Beijing Normal University (Grant No. 12900-312200502405). P. Lin is partially supported by the National Natural Science Foundation of China (Grant No. 12371388) and the Beijing Natural Science Foundation (Grant No. IS25005).


\end{document}